\documentclass[10pt]{csarticle}
\pdfoutput=1
\usepackage{paralist}
\usepackage{hyperref}
\usepackage[small]{caption2}
\usepackage{booktabs}
\usepackage{bbm}
\usepackage{fullpage}

\newcommand{\eps}{\epsilon}

\newcommand{\scr}[1]{\mathcal{#1}}

\makeatletter
\def\timenow{\@tempcnta\time
  \@tempcntb\@tempcnta
  \divide\@tempcntb60
  \ifnum10>\@tempcntb0\fi\number\@tempcntb
  \multiply\@tempcntb60
  \advance\@tempcnta-\@tempcntb:\ifnum10>\@tempcnta0\fi\number\@tempcnta}
\makeatother

\author{L. Addario-Berry\thanks{
    Email: \tt louigi@gmail.com
    \rm .}, N. Broutin\thanks{
    Email: \tt nbrout@cs.mcgill.ca
    \rm .}, C. Goldschmidt\thanks{
    Email: \tt goldschm@stats.ox.ac.uk
    \rm .}
}

\newcommand{\Probc}[1]{\ensuremath{\mathbb{P} (#1)}}
\newcommand{\Cprobc}[2]{\mathbb{P}(#1 \; |  \; #2 )}
\newcommand{\Ec}[1]{\ensuremath{\mathbb{E} [#1]}}
\newcommand{\diamc}[1]{\mathrm{diam}(#1)}

\newcommand{\CExpc}[2]{\mathbb{E}[ #1 \;| \; #2 \;]}

\newcommand{\oDFS}{{{\bf oDFS}}}
\newcommand{\so}{\mathcal{O}}
\newenvironment{steplist}
{ \begin{list}%
	{$\bullet$}%
	{\setlength{\labelwidth}{100pt}%
	 \setlength{\leftmargin}{80pt}%
	 \setlength{\itemsep}{\parsep}}}%
{ \end{list} }

\newcommand{\gh}{\mathrm{GH}}
\newcommand{\diam}[1]{\mathrm{diam}\!\left(#1\right)}
\newcommand{\dgh}{d_{\gh}}
\newcommand{\eqdist}{\ensuremath{\stackrel{d}{=}}}
\newcommand{\ExpC}[2]{\mathbb{E}\event{#1 \left| \; #2 \right. \;}}

\renewcommand{\E}[1]{\ensuremath{\mathbb{E} \left[#1 \right]}}
\newcommand{\Prob}[1]{\ensuremath{\mathbb{P} \left(#1 \right)}}
\renewcommand{\I}[1]{\ensuremath{\mathbbm{1}_{ \{ #1 \} }}}

\renewcommand{\fl}[1]{\ensuremath{\lfloor #1 \rfloor}}
\renewcommand{\ce}[1]{\ensuremath{\lceil #1 \rceil}}

\renewcommand{\subset}{\subseteq}
\newcommand{\convdist}{\ensuremath{\stackrel{d}{\rightarrow}}}
\newcommand{\convprob}{\ensuremath{\stackrel{p}{\rightarrow}}}

\newcommand{\equidist}{\ensuremath{\stackrel{d}{=}}}

\hypersetup{
    bookmarks=true,         
    unicode=false,          
    pdftoolbar=true,        
    pdfmenubar=true,        
    pdffitwindow=true,      
    pdftitle={My title},    
    pdfauthor={Author},     
    pdfsubject={Subject},   
    pdfnewwindow=true,      
    pdfkeywords={keywords}, 
    colorlinks=true,       
    linkcolor=blue,          
    citecolor=blue,        
    filecolor=blue,      
    urlcolor=blue           
}

\begin{document}

\title{The continuum limit of critical random graphs}
\author{L. Addario-Berry \and N. Broutin \and C. Goldschmidt}
\maketitle

\begin{abstract}We consider the Erd\H{o}s--R\'enyi random graph $G(n,p)$ inside the critical window, that is when $p=1/n+\lambda n^{-4/3}$, for some fixed $\lambda\in \R$. Then, as a metric space with the graph distance rescaled by $n^{-1/3}$, the sequence of connected components $G(n,p)$ converges towards a sequence of continuous compact metric spaces. The result relies on a bijection between graphs and certain marked random walks, and the theory of continuum random trees. Our result gives access to the answers to a great many questions about distances in critical random graphs. In particular, we deduce that the diameter of $G(n,p)$ rescaled by $n^{-1/3}$ converges in distribution to an absolutely continuous random variable with finite mean.   

\vspace{0.2cm}
\end{abstract}

\section{Introduction}\label{sec:intro}

\subsection*{Random graphs and the phase transition} 

Since its introduction by \citet{erdos60evolution}, the model $G(n,p)$ of random graphs has received an enormous amount of attention \cite{janson00random, Bollobas2001}. In this model, a graph on $n$ labeled vertices $\{1,2,\dots, n\}$ is chosen randomly by joining any two vertices by an edge with probability $p$, independently for different pairs of vertices. This model exhibits a radical change in structure (or \emph{phase transition}) for large $n$ when $p=p(n)\sim 1/n$.  For $p\sim c/n$ with $c<1$, the largest connected component has size (number of vertices) $O(\log n)$.  On the other hand, when $c > 1$, there is a connected component containing a positive proportion of the vertices (the \emph{giant component}). The cases $c<1$ and $c>1$ are called \emph{subcritical} and \emph{supercritical} respectively. This phase transition was discovered by Erd\H{o}s and R\'enyi in their seminal paper \cite{erdos60evolution}; indeed, they further observed that in the \emph{critical} case, when $p= 1/n$, the largest components of $G(n,p)$ have sizes of order $n^{2/3}$. For this reason, the phase transition in random graphs is sometimes dubbed the \emph{double jump}.

Understanding the critical random graph (when $p=p(n)\sim 1/n$) requires a different and finer scaling: the natural parameterization turns out to be of the form $p=p(n)=1/n+\lambda n^{-4/3}$, for $\lambda = o(n^{1/3})$ \cite{bollobas84evolution,Luczak1990, LuPiWi1994}. In this paper, we will restrict our attention to $\lambda \in \R$; this parameter range is then usually called the \emph{critical window}.  One of the most significant results about random graphs in the critical regime was proved by \citet{aldous97brownian}.  He observed that one could encode various aspects of the structure of the random graph (specifically, the sizes and surpluses of the components) using stochastic processes.  His insight was that standard limit theory for such processes could then be used to get at the relevant limiting quantities, which could, moreover, be analyzed using powerful stochastic-process tools.  Fix $\lambda \in \R$, set $p=1/n+\lambda n^{-4/3}$ and write $Z_i^n$ and $S_i^n$ for the size and surplus (that is, the number of edges which would need to be removed in order to obtain a tree) of $\mathcal{C}_i^n$, the $i$-th largest component of $G(n,p)$.  Set $\mathbf Z^n=(Z_1^n, Z_2^n, \dots)$ and $\mathbf S^n=(S_1^n, S_2^n, \dots)$. 
\begin{thm}[\citet{aldous97brownian}]  \label{thm:aldousthm}
As $n\to\infty$.
\[
(n^{-2/3}\mathbf Z^n, \mathbf S^n)\convdist (\mathbf Z, \mathbf S).
\]
\end{thm}
Here, the convergence of the first co-ordinate takes place in $\ell^2_{\searrow}$, the set of infinite sequences $(x_1,x_2,\dots)$ with $x_1\ge x_2\ge \dots \ge 0$ and $\sum_{i\ge 1} x_i^2<\infty$. (See also \cite{LuPiWi1994,JaSp2007}.) The limit $(\mathbf Z, \mathbf S)$ is described in terms of a Brownian motion with parabolic drift, $(W^{\lambda}(t), t \ge 0)$, where
\[
W^\lambda(t):=W(t)+t\lambda-\frac{t^2}2,
\]
and $(W(t), t\ge 0)$ is a standard Brownian motion. The limit $\mathbf Z$ has the distribution of the ordered sequence of lengths of excursions of the reflected process $W^\lambda(t)-\min_{0\le s\le t}W^\lambda (s)$ above 0, while $\mathbf S$ is the sequence of numbers of points of a Poisson point process with rate one in $\R^+\times \R^+$ lying under the corresponding excursions. Aldous's limiting picture has since been extended to ``immigration'' models of random graphs \cite{ap00}, hypergraphs \cite{cg05}, and most recently to random regular graphs with fixed degree \cite{nachperreg2007}. 

The purpose of this paper is to give a precise description of the limit of the sequence of \emph{components} $\mathbf{\mathcal{C}}^{n} = (\mathcal{C}_1^{n}, \mathcal{C}_2^{n}, \ldots)$.  Here, we view $\mathcal{C}_1^{n}, \mathcal{C}_2^{n}, \ldots$ as metric spaces $M_1^{n}, M_2^{n}, \ldots$, where the metric is the usual graph distance, which we rescale by $n^{-1/3}$.  The limit object is then a sequence of compact metric spaces $\mathbf{M}=(M_1,M_2, \dots)$.  
The appropriate topology for our convergence result is that generated by the Gromov--Hausdorff distance on the set of compact metric spaces, which we now define.  Firstly, for a metric space $(M,\delta)$, write $d_H$ for the Hausdorff distance between two compact subsets $K, K'$ of $M$, that is
\[
d_H(K,K') = \inf\{\epsilon > 0: K \subset F_{\epsilon}(K') \text{ and } K' \subset F_{\epsilon}(K)\},
\]
where $F_{\epsilon}(K) := \{x \in M: \delta(x,K) \leq \epsilon\}$ is the $\epsilon$-fattening of the set $K$.  Suppose now that $X$ and $X'$ are two compact metric spaces, each ``rooted" at a distinguished point, called $\rho$ and $\rho'$ respectively.  Then we define the \emph{Gromov--Hausdorff} distance between $X$ and $X'$ to be
\[
d_{GH}(X,X') = \inf\{d_H(\phi(X), \phi'(X')) \vee \delta(\phi(\rho), \phi(\rho'))\}
\]
where the infimum is taken over all choices of metric space $(M,\delta)$ and all isometric embeddings $\phi: X \to M$ and $\phi':X' \to M$. (Throughout the paper, when viewing a connected labeled graph $G$ as a metric space, we will consider G to be rooted at its vertex of smallest label.) The main result of the paper is the following theorem. 

\begin{thm}\label{thm:main}
As $n \to \infty$,
\[
(n^{-2/3} \mathbf{Z}^{n}, n^{-1/3} \mathbf{M}^{n}) \convdist (\mathbf{Z}, \mathbf{M}),
\]
for an appropriate limiting sequence of metric spaces $\mathbf{M} = (M_1, M_2, \ldots)$. 
Convergence in the second co-ordinate here is in the metric specified by
\begin{equation}\label{4metric}
d(\mathbf{A}, \mathbf{B}) = \left( \sum_{i=1}^{\infty} d_{\mathrm{GH}}(A_i, B_i)^4 \right)^{1/4}
\end{equation}
for any sequences of metric spaces $\mathbf{A} = (A_1, A_2, \ldots)$ and $\mathbf{B} = (B_1, B_2, \ldots)$.
\end{thm}
We will eventually state and prove a more precise version of this theorem, Theorem \ref{thm:4metric}, once we have introduced the appropriate limiting sequence. 
For the moment, we will remark only that convergence in the distance defined above implies convergence of distances in the graph.  As a result, we will be able to give a precise answer to a great many asymptotic distance problems for $G(n,p)$ inside the critical window.
Before we can give an intuitive description of our limit object, we need to introduce one of its fundamental building blocks: the continuum random tree.

\subsection*{The continuum random tree}

In recent years, a huge literature has grown up around the notion of \emph{real trees}.  Here we will concentrate on the most famous random example of such trees, Aldous' Brownian continuum random tree (see \cite{aldous91crt1,aldous91crt2,aldous93crt3}), and encourage the interested reader to look at \cite{evans08probrealtrees,duquesnelegall02randomtrees,legall05survey} and the references therein for more general cases.

The fundamental idea is that continuous functions can be used to encode tree structures.  We will begin our discussion by considering a rooted combinatorial tree on $n$ vertices labeled by $[n] := \{1,2\ldots,n\}$.  There are three (somewhat different) encodings of such a tree which will be useful to us.  We will introduce two of them here and explain the third, which plays a more technical role, in the main body of the paper (see Section \ref{dfs}).  For both of the encodings we will discuss here, we need to introduce the notion of the \emph{depth-first ordering} of the vertices.  For each vertex $v$, there is a unique path from $v$ to the root, $\rho$.  Call the vertices along this path the \emph{ancestors} of $v$.  Relabel each vertex by the string which consists of the concatenation of all the labels of its ancestors and its own label, so that if the path from $\rho$ to $v$ is $\rho, a_1, a_2, \ldots, a_m, v$, relabel $v$ by the string $\rho a_1 a_2 \ldots a_m v$.  The depth-first ordering of the vertices is then precisely the lexicographical ordering on these strings.  More intuitively, we look first at the root, then at its lowest-labeled child, then at the lowest-labeled child of that vertex, and so on until we hit a leaf.  Then we backtrack one generation and look at the next lowest-labeled child and its descendents as before.  See Figure \ref{fig:dforder}.

\begin{figure}
\centering
\begin{picture}(420,110)
\put(60,0){\includegraphics[scale=.75]{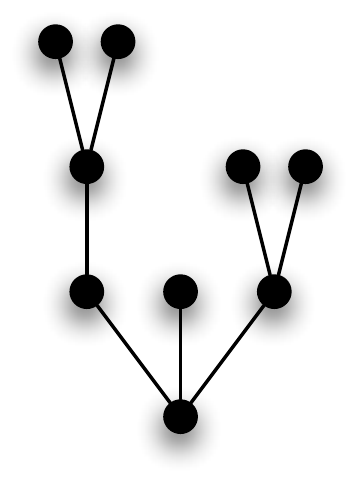}}
\put(105,5){$1$}
\put(69,42){$3$}
\put(90,42){$5$}
\put(110,42){$8$}
\put(69,69){$9$}
\put(102,69){$6$}
\put(132,69){$7$}
\put(60,95){$2$}
\put(93,95){$4$}
\put(270,0){\includegraphics[scale=.75]{tree}}
\put(315,5){$1$}
\put(279,42){$2$}
\put(300,42){$6$}
\put(320,42){$7$}
\put(279,69){$3$}
\put(312,69){$8$}
\put(342,69){$9$}
\put(270,95){$4$}
\put(302,95){$5$}
\end{picture}
\label{fig:dforder}
\caption{Left: a tree on $[9]$.  Right: the same tree but labeled in depth-first order.}
\end{figure}

The first encoding is in terms of the \emph{height function} (or, when the tree is random, the \emph{height process}).  For $0 \leq i \leq n-1$, let $H(i)$ be the graph distance from the root of the $(i+1)$-st vertex visited in depth-first order (so that $H(0) = 0$, since we always start from the root).  Then the height function of the tree is the discrete function $(H(i), 0 \leq i \leq n-1)$.  (See Figure \ref{fig:heightcontour}.)  (It is, perhaps, unfortunate that we talk about a \emph{depth}-first ordering and vertices having \emph{heights}.  The reason for this is that the two pieces of terminology originated in different communities.  However, since both are now standard, we have chosen to keep them and hope that the reader will forgive the ensuing clumsiness.)

\begin{figure}
\centering
\begin{picture}(420,120)
\put(-21,2){\includegraphics[scale=.65]{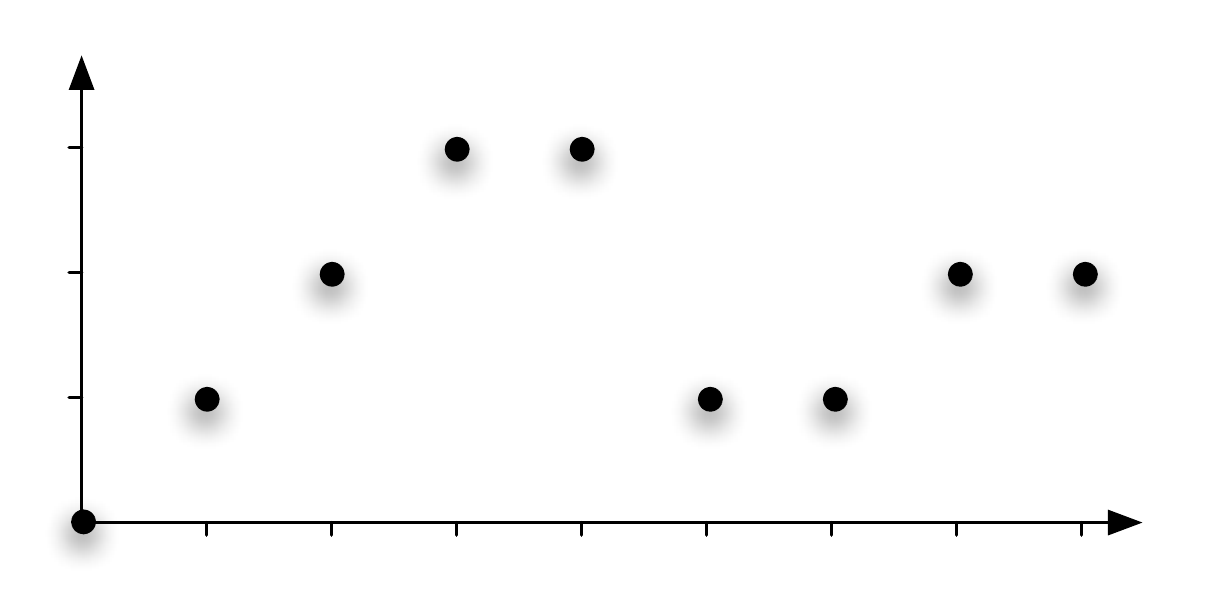}}
\put(16,0){$1$}
\put(40,0){$2$}
\put(63,0){$3$}
\put(87,0){$4$}
\put(110,0){$5$}
\put(133,0){$6$}
\put(157,0){$7$}
\put(180,0){$8$}
\put(196,0){$i$}
\put(-20,40){$1$}
\put(-20,64){$2$}
\put(-20,88){$3$}
\put(-20,121){$H(i)$}
\put(220,10){\includegraphics[scale=.65]{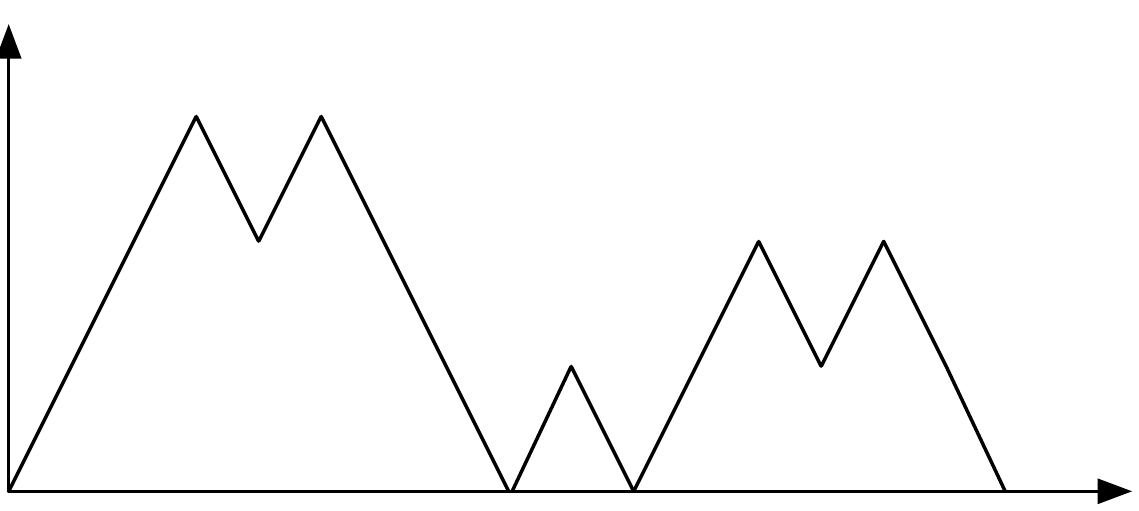}}
\put(210,115){$C(t)$}
\put(430,0){$t$}
\end{picture}
\caption{The height process and contour process associated with the tree in Figure \ref{fig:dforder}.}
\label{fig:heightcontour}
\end{figure}
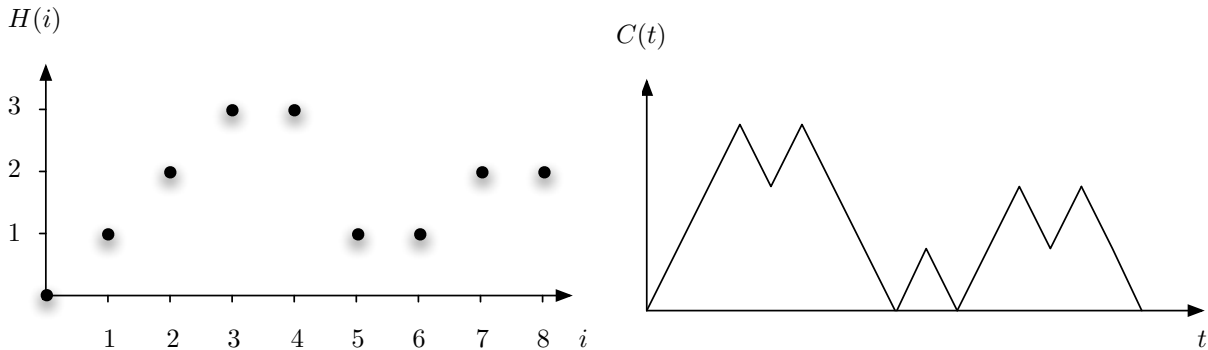

The second encoding is via the \emph{contour function} (or, when the tree is random, the \emph{contour process}).  Starting at the root, trace the contour of the tree in depth-first order at speed 1 and record the current distance from the root, $C(t)$ at time $t$.  That is, instead of just hopping from vertex to vertex, take the shortest tree path between them at speed 1.  (See Figure \ref{fig:heightcontour}.)  It is easy to see that every edge is now traversed twice, and so this produces a function $(C(t), 0 \leq t \leq 2(n-1))$.

Note that the topology of the tree, but not the labels, can be recovered from either the height or contour function.

Suppose now that we take a uniform random tree on $[n]$, rooted at 1.  Let $(H^n(i), 0 \leq i \leq n-1)$ be its height process and $(C^n(t), 0 \leq t \leq 2(n-1))$ its contour process.

\begin{thm}[\citet{aldous93crt3}] \label{thm:aldousheightcontour}
Let $(e(t), 0 \leq t \leq 1)$ be a standard Brownian excursion. Then, as $n \to \infty$,
\begin{gather*}
\frac{1}{\sqrt{n}}(H^n(\fl{nt}), 0 \leq t \leq 1)  \convdist 2(e(t), 0 \leq t \leq 1) 
\qquad\mbox{and}\qquad
\frac{1}{\sqrt{n}}(C^n(nt), 0 \leq t \leq 2)\convdist 2(e(t/2), 0 \leq t \leq 2).
\end{gather*}
\end{thm}

Here, convergence is in the space $\mathbb{D}([0,1],\R^+)$ of non-negative c\`adl\`ag functions (right-continuous with left limits), equipped with the Skorohod topology (see, for example, \citet{Billingsley1968}).

In fact, this convergence turns out to imply that the \emph{tree itself} converges, in a sense which we will now make precise.  We follow the exposition of \citet{legall05survey}.  Take a uniform random tree on $[n]$ and view it as a path metric space $T_n$ by taking the union of the line segments joining the vertices, each assumed to have length 1.  (Note that the original tree-labels no longer play any role, except that we will think of the metric space $T_n$ as being rooted at the point corresponding to the old label 1.)  Then the distance between two elements $\sigma$ and $\sigma'$ of $T_n$ is simply the length of the shortest path between them.  We will abuse notation somewhat and write $n^{-1/2}T_n$ for the same metric space with all distances rescaled by $n^{-1/2}$.

In order to state the convergence result, we need to specify the limit object.  We will start with some general definitions.

A compact metric space $(\mathcal{T},d)$ is a \emph{real tree} if for all $x,y \in \mathcal{T}$
\begin{itemize}
\item there exists a unique geodesic from $x$ to $y$ i.e.\ there exists a unique isometry $f_{x,y}: [0, d(x,y)] \to \mathcal{T}$ such that $f_{x,y}(0) = x$ and $f_{x,y}(d(x,y)) = y$.  The image of $f_{x,y}$ is called $[[x,y]]$;
\item the only non-self-intersecting path from $x$ to $y$ is $[[x,y]]$ i.e.\ if $q: [0,1] \to \mathcal{T}$ is continuous and injective and such that $q(0) = x$ and $q(1) = y$ then $q([0,1]) = [[x,y]]$.
\end{itemize}
An element $x \in \mathcal{T}$ is called a \emph{vertex}.  A \emph{rooted real tree} is a real tree $(\mathcal{T},d)$ with a distinguished vertex $\rho$ called the \emph{root}.  The \emph{height} of a vertex $v$ is $d(\rho,v)$.  By a \emph{leaf}, we mean a vertex $v$ which does not belong to $[[\rho,w]]$ for any $w \neq v$.  Write $\mathcal{L}(\mathcal{T})$ for the set of leaves of $\mathcal{T}$.  Finally, write $[[x,y[[$ for $f_{x,y}([0, d(x,y)))$.

Suppose that $h: [0,\infty) \to [0, \infty)$ is a continuous function of compact support such that $h(0) = 0$.  Use it to define a pseudo-metric $d$ by
\[
d(x,y) = h(x) + h(y) - 2 \inf_{x\wedge y \leq t \leq x \vee y} h(t), \qquad x,y \in [0, \infty).
\]
Let $x \sim y$ if $d(x,y) = 0$, so that $\sim$ is an equivalence relation.  Let $\mathcal{T} = [0,\infty)/\sim$ and denote by $\tau: [0,\infty) \to \mathcal{T}$ the canonical projection.  If $\sigma$ is the supremum of the support of $h$ then note that $\tau(s) = 0$ for all $s \geq \sigma$.  This entails that $\mathcal{T} = \tau([0,\sigma])$ is compact.  The metric space $(\mathcal{T},d)$ can then be shown to be a real tree.  
Set $\rho = \tau(0)$ and take $\rho$ to be the root.
  
Now take 
\[
h(t) = \begin{cases}
          2e(t) & 0 \leq t \leq 1, \\
          0 & t > 1,
          \end{cases}
\]
where, as in Theorem \ref{thm:aldousheightcontour}, and for the rest of the paper, $(e(t), 0 \le t \le 1)$ is a standard Brownian excursion.  Then the resulting tree is the \emph{Brownian continuum random tree} (or \emph{CRT}, when this is unambiguous).  Note that the role of the Brownian excursion (multiplied by 2) is that of both height and contour process.  We will always think of the CRT as rooted. We then have the following. 

\begin{thm}[\citet{aldous93crt3,legall06realtrees}]
Let $T_n$ be the metric space corresponding to a uniform random tree on $[n]$ and let $\mathcal{T}$ be the CRT. Then
\[
n^{-1/2} T_n \convdist \mathcal{T},
\]
as $n \to \infty$, where convergence is in the Gromov--Hausdorff sense.
\end{thm}

It is perhaps useful to note here that the limit tree $\mathcal{T}$ comes equipped with a \emph{mass measure} $\mu$, which is simply the probability measure induced on $\mathcal{T}$ from Lebesgue measure on $[0,1]$.  Unsurprisingly, $\mu$ is the limit of the empirical measure on the uniform tree on $[n]$ which puts mass $1/n$ on each vertex.  Later on, we will use the fact that $\mu(\mathcal{L}(\mathcal{T})) = 1$, i.e., $\mu$ is concentrated on the leaves of the CRT \cite[][p. 60]{aldous91crt2}.

\subsection*{The limit of a critical random graph} 

In this section we give a non-technical description of the limiting object in Theorem \ref{thm:main}. 
Conditional on their size and surplus, components of $G(n,p)$ are uniform connected graphs with that size and surplus.  Moreover, as we have discussed, in the critical window, where $p = n^{-1} + \lambda n^{-4/3}$ for some $\lambda \in \mathbb{R}$, the largest components have size of order $n^{2/3}$ and surplus of constant order.  In order to understand better the structure of these components, we look at uniform connected graphs with ``small" surplus.  For definiteness, we will consider a uniform connected graph on $m$ vertices with surplus $s$.

Such connected graphs always possess spanning subtrees.  A particular one of these, which we will refer to as the \emph{depth-first tree}, will be very useful to us.  As its name suggests, this tree is constructed via a depth-first search procedure (which we will not detail until later).  The depth-first tree of a uniform random connected graph with $s$ surplus edges is \emph{not} a uniform random tree, but has a ``tilted'' distribution which is biased (in a way depending on $s$) in favor of trees with large \emph{area} (for ``typical" trees, this is essentially the sum of the heights of the vertices of the tree). We will define this tilting precisely later, and will then spend much of the paper studying it. 

The limit of a uniform random tree on $m$ vertices, thought of as a metric space with graph distances rescaled by $m^{-1/2}$,  is the continuum random tree.  It turns out that the limit of the depth-first tree associated with a connected component with surplus $s$, with the same rescaling, is a continuum random tree coded by a Brownian excursion whose distribution is biased in favor of excursions having large area (where area now has its habitual meaning; once again, the bias depends on $s$).

The difference between the depth-first tree and the connected graph is precisely the $s$ surplus edges.  The depth-first tree is convenient because not only can we describe its continuum limit but, given the tree, it is straightforward to describe where surplus edges may go (we call such locations \emph{permitted edges}).  Indeed, the surplus edges are equally likely to be any of the possible $s$-sets of permitted edges.  A careful analysis of the locations of the surplus edges in the finite graph leads to the following surprisingly simple limit description for a uniform random connected graph with surplus $s$.  Take a continuum random tree with tilted distribution and independently select $s$ of its leaves with density proportional to their height.  For each of these leaves there is a unique path to the root of the tree.  Independently for each of the $s$ leaves, pick a point uniformly along the path and identify the leaf and the selected point.  (Note that we identify the points because edge-lengths have shrunk to 0 in the limit.)

Having thus described the limit of a \emph{single} component of a critical random graph, it remains to describe the limit of the \emph{collection} of components.  The key is Aldous' description of the limiting sizes and surpluses of the components in terms of the excursions above 0 of the reflected Brownian motion with parabolic drift.  The excursion lengths give the limiting component sizes.  The auxilliary Poisson process of points with unit intensity under the graph of the reflected process gives the limit of the numbers of surplus edges.  In fact, more is true.  The excursions themselves can be viewed as coding the sequence of limits of depth-first trees of the components; the locations of the Poisson points under the excursions can be seen to correspond in  a natural way to the locations of the surplus edges.  Intuitively, the successive excursions are selected according to an inhomogeneous excursion measure associated with the process.  Under this measure, the length and area of an excursion are related in precisely the correct ``tilted'' manner, so that, conditional on an excursion having length $\sigma$ and $s$ Poisson points, the metric space it codes has precisely the distribution of the limit of a uniform random connected graph on $\sim \sigma n^{2/3}$ vertices with $s$ surplus edges, whose edge-lengths have been rescaled by $n^{-1/3}$.

\subsection*{The diameter of random graphs} 

The \emph{diameter} of a connected graph is the largest distance between any pair of vertices of the graph. 
For a general graph $G$, we define the diameter of $G$ to be the greatest distance between any pair of vertices lying in the same connected component. 

The behavior of the diameter of $G(n,p)$ for $p=O(1/n)$ is a pernicious problem for which few detailed results were known until extremely recently \citep{RiWo2008}. 
(For references on distances in dense graphs $G(n,p)$ with $p$ fixed, see \cite{Bollobas2001}.) 
In the subcritical phase, when $p=p(n)=1/n+\lambda(n)n^{-4/3}$ and $\lambda\to-\infty$, \citet{Luczak1998} showed that the diameter of $G(n,p)$ is within one of the largest diameter of a tree component with probability tending to one. \citet{ChLu2001} focused on the early supercritical phase, when $np>1$ and $np\le c \log n$. (Problem 4 in their paper asks about the diameter of $G(n,p)$ inside the critical window.) More recently, \citet{RiWo2008} have addressed the problem for the range $p=c/n$, $c>1$ fixed, proving essentially best possible bounds on the behavior of the diameter for such $p$. They also have results for $p=1/n + \lambda n^{-4/3}$, with $\lambda=o(n^{1/3})$ which are also essentially optimal, 
but their error terms require that $\lambda =\lambda(n) \geq e^{(\log^* n)^4)} \rightarrow \infty$ as $n \rightarrow \infty$. When $\lambda$ is fixed, $G(n,p)$ contains several complex (i.e., with multiple cycles) components of comparable size, and any one of them has a non-vanishing probability of accounting for the diameter. \citet{NaPe2008} have shown that the greatest diameter of any connected component of $G(n,p)$ is with high probability $\Theta(n^{1/3})$ for $p$ in this range; this result also 
follows trivially from work of \citet*{addarioberry2006dms,addario08diam} on the diameter of the minimum spanning tree of a complete graph in which each edge $e$ has an independent uniform $[0,1]$ edge weight. 

In this paper, we demonstrate how Theorem \ref{thm:main} allows us to straightforwardly derive precise results on the diameter of $G(n,p)$ for $p=1/n+\lambda n^{-4/3}$ with $\lambda$ fixed. 
(In a companion paper \citep{AdBrGo2009b}, we use the limit theory developed in this paper to answer many other questions about the distribution of distances between vertices in critical random graphs.) 

\begin{thm}\label{thm:diam_gnp}
Suppose that $p=1/n+\lambda n^{-4/3}$ for $\lambda\in \R$. Let $D^{n}_i$ denote the diameter of the $i$-th largest connected component of $G(n,p)$. Let $D^{n}=\sup_i D^{n}_i$ denote the diameter of $G(n,p)$ itself.
For each fixed $i$ there is a random variable $D_i\geq 0$ with  $\E{D_i} < \infty$ such 
that 
$$n^{-1/3}D_i^{n} \convdist D_i.$$  
Furthermore there is a random variable $D \geq 0$ with an absolutely continuous distribution and $\E{D} < \infty$ such that $n^{-1/3} D^n \convdist D$. 
\end{thm}

\subsection*{Plan of the paper}

The depth-first procedure is presented in Section~\ref{dfs}.  Given a connected component of $G(n,p)$ of size $m$ it yields a ``canonical'' spanning tree $\tilde T_m^p$. The distribution of the tree $\tilde T_m^p$ is studied in Section~\ref{sec:tilted_trees}. In Section~\ref{sec:connected_components}, we describe the graphs obtained by adding random surplus edges to the trees $\tilde T_m^p$ and introduce the continuous limit of connected components conditional on their size. Finally, very much as $G(n,p)$ may be obtained by taking a sequence of connected components which are independent given their sizes, the continuum limit of $G(n,p)$ can be constructed by first setting the sizes of the components to have the correct distribution, and then generating components independently. This is described in Section~\ref{sec:random_graph_limit}.


\section{Depth-first search and random graphs}
\label{dfs}

Let $G=(V,E)$ be an undirected graph with $V=[n]=\{1,\ldots,n\}$. The {\em ordered depth-first search forest} for $G$ (started from vertex $k$) 
is the spanning forest of $G$ obtained by running depth-first search (DFS) on $G$, using the rule that 
whenever there is a choice of which vertex to explore, the smallest-labeled vertex is always explored first.
For clarity, we explain more precisely what we mean by this description, and introduce some relevant notation. 
For $i\ge 0$, we define the ordered set (or \emph{stack} \cite[see][]{CoLeRiSt2001}) $\so_i$ of open vertices at time $i$, and the set $\mathcal A_i$ of the vertices that have already been explored at time $i$. We say that a vertex $u$ has been \emph{seen} at time $i$ if $u\in \so_i \cup \mathcal A_i$.  Let $c_i$ be a counter which keeps track of how many components have been discovered up to time $i$.

\begin{steplist}
	\item[{\sc {\bf oDFS}}($G$)\hspace{24pt}]
	\item[{\sc Initialization}] Set $\so_0=(1)$, $\mathcal A_0=\emptyset$, $c_0 = 1$. 
	\item[{\sc Step $i$}\hspace{38pt}] ($0 \leq i \leq n-1$): Let $v_{i}$ be the first vertex of $\so_i$ and let $\mathcal N_i$ be the set of neighbors of $v_i$ in $[n]\setminus (\mathcal A_i\cup \so_i)$. Set $\mathcal A_{i+1}= \mathcal A_i\cup \{v_i\}$.  Construct $\so_{i+1}$ from $\so_i$ by removing $v_i$ from the start of $\so_i$ and affixing the elements of $\mathcal N_i$ in increasing order to the start of $\so_i\setminus\{v_i\}$.  If now $\so_{i+1} = \emptyset$, add to it the lowest-labeled element of $[n] \setminus \mathcal A_{i+1}$ and set $c_{i+1} = c_i +1$.  Otherwise, set $c_{i+1} = c_i$.
\end{steplist} 
After step $n-1$, we have $\so_{n}=\emptyset$. We remark that this procedure defines a reordering $\{v_0,\ldots,v_{n-1}\}$ of $[n]$, and for any $G$, $\oDFS(G)$ always sets $v_0=1$. 
We refer to DFS run according to this rule as {\em ordered DFS}.
(The terminology {\em lexicographic-DFS} may seem natural; however, this has been 
given a slightly different definition by \citet{CoKr2008}.)  We note also that we increment the counter $c_i$ precisely when $v_i$ is the last vertex explored in a component, so that $(c_i, 0 \leq i < n)$ really does count components.  (We observe that if, in {\sc Step $i$}, we affix the elements of $\mathcal N_i$ to the \emph{end} of $\so_i \setminus \{v_i\}$ instead of the start, we obtain the \emph{breadth-first ordering} exploited by \citet{aldous97brownian}; we will discuss this further in Section \ref{sec:random_graph_limit}.)

The forest corresponding to $\oDFS(G)$ consists of all edges $xy$ such that for some $i\in\{0,\ldots,n-1\}$, $x$ is the vertex explored at step $i$ (so $x=v_{i}$) and $y \in N_i$.  We refer to this 
as the \emph{ordered depth-first search forest} for $G$.

For $i=0,1,\ldots,n-1$, let $X(i)=|\so_i\setminus \{v_i\}| - (c_i - 1)$. The process $(X(i), 0\le i< n)$ is called the \emph{depth-first walk} of the graph $G$. (The terminology ``walk" may seem odd here, but in the random context which is the focus of this paper, $(X(i), 0 \le i < n)$ turns out to be something like a random walk.) 

We will particularly make use of these ideas in the case where $G$ is connected.  In that situation, we have $c_i = 1$ for all $0 \leq i < n$ and so the algorithm is simpler.  (In particular, the set of open vertices only becomes empty at the end of the procedure.)
Furthermore, since $G$ is connected, the ordered depth-first search forest is now a tree, which we will refer to as the \emph{depth-first tree} and write $T(G)$.  The depth-first walk $(X(i), 0 \le i < n)$ now has the simpler representation $X(i) = |\so_i \setminus \{v_i\}| = |\so_i| - 1$ for $0 \leq i < n$ and can be interpreted as the number of vertices seen but not yet fully explored at step $i$ of the $\oDFS(G)$ procedure.  The following observation will be important later: the vertices in $\so_i \setminus \{v_i\}$ all lie at distance 1 from the path from the root to $v_i$.  Put differently, the vertices of $\so_i \setminus \{v_i\}$ are all younger siblings of ancestors of $v_i$.

We next consider running $\oDFS$ on a tree $T=([n],E)$. Of course, now the depth-first tree will be $T$ itself.  We define the {\em area} of a tree $T$ to be 
\[
a(T)=\sum_{i=1}^{n-1} X(i).
\] 
The area $a(T)$ corresponds to the number of integral points in $\{(i,j): 0\le i< n, 1\le j \le X(i)\}$ (see Figure~\ref{fig:dfs_walk}). For convenience, we also define a continuous interpolation: for $s\in [0,n-1]$ 
\[
X(s):=X(\fl{s})+(s-\fl s)\cdot(X(\fl s +1)-X(\fl s)).
\]
It will sometimes be convenient to see $X$ as a discrete excursion, and we extend $X$ to all of $\R^+$ by setting $X(s)=0$ for all $s\ge n-1$.
Say that an edge $uv \notin E$ is {\em permitted by} $\oDFS(T)$ if, in the $\oDFS(T)$ procedure run 
on $T$, at some stage of the process, $i$ and $j$ are both seen but neither is fully explored: there exists $i\in \{0,\dots,n-1\}$ such that $u,v\in \mathcal O_i$. The following lemma is then straightforward. 

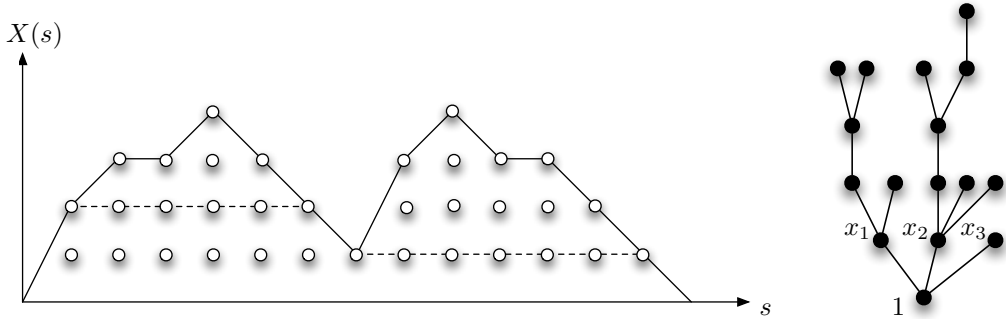
\begin{figure}
\centering
\begin{picture}(400,150)
\put(0,17){\includegraphics[scale=.5]{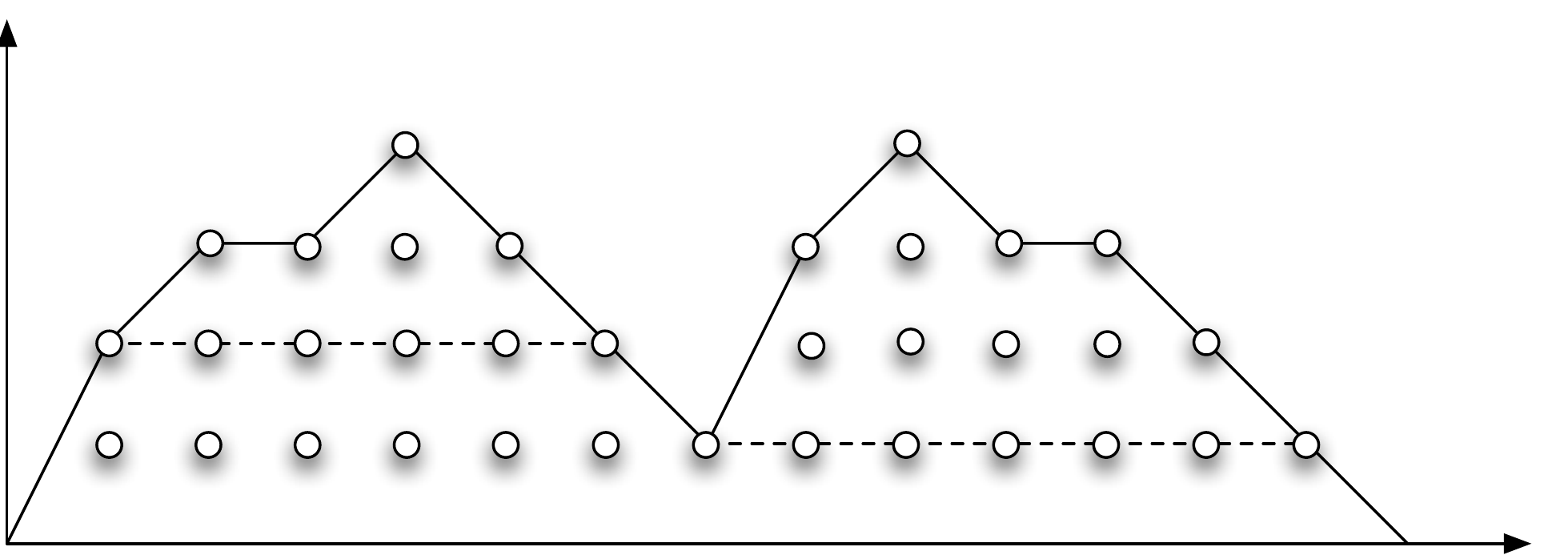}}
\put(300,10){\includegraphics[scale=.6]{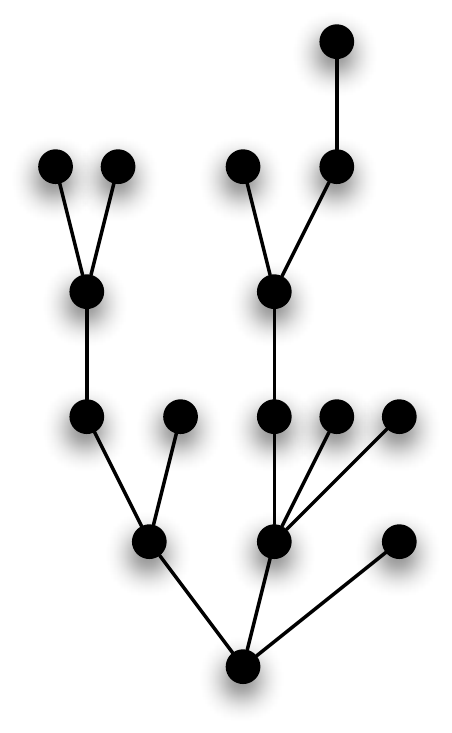}}
\put(-5,118){$X(s)$}
\put(280,15){$s$}
\put(330,15){$1$}
\put(312,45){$x_1$}
\put(334,45){$x_2$}
\put(356,45){$x_3$}
\end{picture}
\caption{\label{fig:dfs_walk}The (interpolated) depth-first walk $X(s)$ of the tree $T$ displayed on the right is shown. We have emphasized the integral points that contribute to the area $a(T)$. The portions of the walk above the dashed lines correspond to the $\oDFS(T_1)$ and $\oDFS(T_2)$ processes (started from $x_1$ and $x_2$ respectively).}
\end{figure}

\begin{lem}\label{area}
The number of edges permitted by $\oDFS(T)$ is precisely $a(T)$.
\end{lem}
\begin{proof}
We proceed by induction on $n$.  For $n=1,2$ the claim is clear. For $n \geq 3$, let $x_1,\ldots,x_i$ be the neighbors of  
$1$ in $T$, listed in increasing order, and let $T_1,\ldots,T_i$ be the trees containing 
$x_1,\ldots,x_i$, respectively, when vertex $1$ is removed from $T$. 

By its definition, the procedure $\oDFS(T)$ simply uncovers $x_1,\ldots,x_i$, 
then runs \oDFS($T_j$), for each $j=1,\ldots,i$, in this this order, but started (exceptionally) from $x_j$ in each case.
In particular, for each $j=1,\ldots,i$, each edge from $x_j$ to $x \in T_k$, $k \leq j$ 
is permitted by $\oDFS(T)$. Thus, the total number of edges with one endpoint in $\{x_1,\ldots,x_i\}$ permitted by $\oDFS(T)$ 
 is precisely 
\[
\sum_{j=1}^i (i-j) |T_{j}|. 
\]
Write $X_T$ and $X_{T_j}$, $1 \le j \le i$ in order to distinguish the depth-first walks on $T$ and on its subtrees.  By induction, it thus follows that the number of edges permitted by 
$\oDFS(T)$ is 
\[
\sum_{j=1}^i ((i-j) |T_{j}| + a(T_{j})) 	
 =  \sum_{j=1}^i \sum_{k=1}^{|T_j|} \pran{(i-j) + X_{T_{j}}(k)} 
 =  \sum_{j=1}^i \sum_{\ell=|T_1|+\ldots+|T_{j-1}|+1}^{|T_1|+\ldots+|T_{j}|} X_T(\ell) 
 = a(T),
\]
since, for $0\le k<|T_j|$ and $\ell=|T_1|+\dots + |T_{j-1}|+k$, the $\ell$-th step of the $\oDFS(T)$ process explores a vertex $v_\ell$ of the tree $T_j$ and $X_T(\ell)=i-j+X_{T_j}(k)$ (see Figure~\ref{fig:dfs_walk}).
\end{proof}

The next lemma characterizes the connected graphs $G$ which have a given depth-first tree. 
This lemma essentially appears in \cite{gessel79dfs}, though that paper uses slightly different terminology and a different canonical vertex ordering for $\oDFS$. 
The correspondence of \citet{Spencer1997} is a precise analog of our lemma when the tree extracted from the connected graph is constructed by breadth-first search rather than depth-first search.  
(Spencer used this correspondence to show that the so-called ``Wright constants'' \citep{wright1977ncs} are essentially factorial weightings of the moments of the area of a standard Brownian excursion.) 
\begin{lem}\label{permit}
Given any tree $T$ and connected graph $G$ on $[n]$, $T(G)=T$ if and only if $G$ can be obtained from $T$ by adding some subset of the edges permitted by $\oDFS(T)$. 
\end{lem} 
\begin{proof} 
First, if $T(G)=T$ then $T$ is certainly a subgraph of $G$. Next, suppose that $G$ can be obtained from $T$ by adding a subset of the edges permitted by $\oDFS(T)$. We proceed 
by induction on $k$, the number of edges of $G$ not in $T$. The case $k=0$ is clear, so suppose $k \geq 1$ and 
let $v_0,\ldots,v_{n-1}$ be the ordering of $[n]$ obtained by running $\oDFS(T)$. Let $v_iv_j$ be the lexicographically least edge of $G$ not in $T$ (written so that $i < j$). 

Now, vertex $v_i$ is explored at step $i$ of $\oDFS(T)$. By our choice of $v_iv_j$, prior to step $i$ the behavior of $\oDFS(T)$ and $\oDFS(G)$ is identical, so in particular 
$\so_{i}(T)=\so_{i}(G)$. Furthermore, since $v_iv_j$ is permitted by $\oDFS(T)$, and $v_j$ is explored after $v_i$, we must have $v_j \in \so_i(T)=\so_i(G)$. 
Thus, $v_iv_j \notin T(G)$, and so $T(G)=T(G\setminus \{v_iv_j\})$. The ``if'' part of the lemma follows by induction. 

Finally, suppose that $G$ contains an edge not permitted by $\oDFS(T)$, and let $v_iv_j$ be the lexicographically least such edge (in the order given by $\oDFS(T)$). 
Then as before, the behavior of $\oDFS(T)$ and $\oDFS(G)$ is necessarily identical prior to step $i$, so in particular $v_j \notin A_i(T)=A_i(G)$. 
Furthermore, since $v_iv_j$ is not permitted by $\oDFS(T)$, $v_j \notin \so_i(T)=\so_i(G)$. But $v_iv_j \in E(G)$, so we will have $v_j \in N_i(G)$ and thus $v_iv_j \in T(G)$. 
Hence, $T(G)\neq T$, which proves the ``only if'' part of the lemma.
\end{proof}

Let $\mathbb{T}_{[n]}$ denote the set of trees on $[n]$ and write $\mathbb G_T$ for the set of connected graphs $G$ with $T(G)=T$. Then it follows from Lemmas~\ref{area} and \ref{permit} that
\[
\{\mathbb G_T:~T \in \mathbb{T}_{[n]}\}
\]
is a partition of the connected graphs on $[n]$, and that the size of $\mathbb G_T$ is $2^{a(T)}$. Recall that the \emph{surplus} of a connected graph $G$ is the minimum number of edges that must be removed in order to obtain a tree, and call it $s(G)$. Then,
for any $k \in \Z^+$, the number of graphs in $\mathbb G_T$ with surplus $k$ is precisely ${a(T)\choose k}$. 
(We interpret ${a \choose k}$ as $0$ if $k > a$ throughout the paper.)  

We use these ideas to give a method of generating a
connected component on a fixed number $m$ of vertices (which, without loss of generality, we will label by $[m]$). Fix $p \in (0,1)$.  First pick a labeled tree $\tilde{T}_m^{p}$ on $[m]$ in such a way that
$\Probc{\tilde{T}_m^{p} = T} \propto (1 - p)^{-a(T)}$.  Now add to $\tilde{T}_m^{p}$ each of the
$a(\tilde{T}_m^{p})$ edges permitted by  $\oDFS(\tilde{T}_m^{p})$ independently with probability $p$,
so that, given $\tilde{T}_m$, we add a $\mathrm{Binomial}(a(\tilde{T}_m), p)$ number of surplus edges.  Call the graph thus
generated $\tilde{G}_m^p$. 
Let $G_m^p$ be a connected component of $G(n,p)$ conditioned to have size (number of vertices) $m$.  
  
\begin{prop}\label{prop:dist_conn_comp}
For any $p \in (0,1)$ and $m \leq n$, $\tilde{G}_m^p$ has the same distribution as $G_m^p$, a connected component of $G(n,p)$ conditioned to have size $m$ (and vertices labeled by $[m]$). 
\end{prop}

\begin{proof}For a graph $G$ on $[m]$, we write $s(G)=|E(G)|-(m-1)$, so if $G$ is connected
then $s(G)$ is its surplus. Consider $G_m^p$ as relabeled with $[m]$ in the unique way that conserves the ordering of vertices. It is then immediate from the definition of $G(n,p)$ that, for a connected graph $G$ on $[m]$,
\[
\p{G_m^p=G} \propto \p{G(m,p)=G}=p^{m-1+s(G)} (1-p)^{\binom m2-m+1 -s(G)}\propto p^{s(G)} (1-p)^{-s(G)}.
\]
Also, by its definition,
\[
\pc{\tilde{G}_{m}^p=G} \propto (1-p)^{-a(T)}
\Cprobc{\tilde{G}_{m}^p=G}{T(G)=T} = (1-p)^{-a(T)}
p^{s(G)}(1-p)^{a(T)-s(G)},
\]
which completes the proof.
\end{proof}

\begin{cor}
Conditional on $s(\tilde G_m^p) = k\ge 0$, $\tilde{G}_m^p$ is a uniformly chosen connected graph on $[m]$ with $m+k-1$ edges, irrespective of the value of $p\in (0,1)$.
\end{cor}


\section{Tilted trees and tilted excursions}
\label{sec:tilted_trees}

In the introduction, we observed that twice the standard Brownian excursion appears as the limit of the height process $(H^n(i), 0 \le i \le n-1)$ and of the contour process $(C^n(t), 0 \le t \le 2(n-1))$ of a uniform random tree on $[n]$ 
(see, e.g., \cite{aldous93crt3, legall05survey}).
Let $(X^n(i), 0 \le i \le n-1)$ be the corresponding depth-first walk.  Then
\[
\frac{1}{\sqrt{n}} (X^n(\fl{nt}), 0 \leq t \leq 1) \convdist (e(t), 0 \le t \le 1),
\]
as $n \to \infty$, with convergence in $\mathbb D([0,1],\R^+)$ equipped with the Skorohod topology (see \citet{MaMo2003}).  (Note that the results in \cite{MaMo2003} are stated in the more general situation of an ordered Galton--Watson tree with an arbitrary finite-variance offspring distribution, conditioned to have $n$ vertices.  If the offspring distribution is taken to be Poisson mean $1$ then the conditioned tree has precisely the metric structure of a uniform labeled tree.)  It is no coincidence that the limits of these three processes should be the same: they are not only the same in distribution, but are actually the \emph{same} excursion.

\begin{thm}[\citet{MaMo2003}]\label{mamothm}
As $n \to \infty$,
\[
\frac{1}{\sqrt{n}}(X^n(\fl{n \,\cdot}), H^n(\fl{n \,\cdot}, C^n(n \,\cdot)) \convdist (e(\,\cdot\,), 2e(\,\cdot\,), 2e(\,\cdot\,/2)).
\]
\end{thm}

We will make considerable use of this fact.  
An essential tool in what follows will be the following estimate on the distance between the depth-first walk and the height process (Theorem 3 of \cite{MaMo2003}).

\begin{thm}[\citet{MaMo2003}] \label{thm:MaMobound}
For any $\nu > 0$, there exist $n_{\nu}$ and $\gamma > 0$ such that, for all $n \geq n_{\nu}$,
\[
\Prob{\sup_{0 \leq i < n} \left| X^n(i) - \frac{H^n(i)}{2} \right| \geq n^{1/4 + \nu}} \leq e^{-\gamma n^\nu}.
\]
\end{thm}

In this section, we focus on understanding the distribution of the \emph{tilted trees} $\tilde T_m^p$. Note that in the case of critical $G(n,p)$ the largest components have size $m$ of order $n^{2/3}$ and $p\sim 1/n$, so that we shall take $p=p(m)$ of order $m^{-3/2}$. We write $\tilde X^m=(\tilde X^m(i), 0\le i<m)$ and $\tilde H^m=(\tilde H^m(i), 0\le i<n)$ for the depth-first walk and height process of $\tilde T_m^p$ (in $\oDFS$ order). Although it is usually impossible to reconstruct the labelling from either $\tilde X^m$ or $\tilde H^m$, the structure of the trees (as unlabeled rooted ordered trees) can be recovered from either $\tilde X^m$ or $\tilde H^m$. We start with a description of the scaling limit of these discrete excursions, which is closely related to the scaling limit of the corresponding processes, $X^m$ and $H^m$, for uniform trees.

Write $\mathcal{E}$ for the space of excursions; that is,
\begin{equation}\label{excursionspace}
\mathcal{E} = \{f \in C(\R^+,\R^+): f(0) = 0, \exists \ \sigma < \infty \text{ such that } f(x) > 0 \ \forall \ x \in (0,\sigma) ~\mbox{and}~ f(x) = 0 \ \forall x \ge \sigma\}.
\end{equation}
Given a function $f \in C([0,\sigma],\R^+)$ with $f(0)=f(\sigma)=0$ and $f(x)>0 \ \forall \ x \in (0,\sigma)$, we will abuse notation by identifying $f$ with the function $g \in \mathcal E$ which has 
$g(x)=f(x)$, $0 \leq x < \sigma$ and $g(x)=0$, $x \geq \sigma$. 
The distance of interest for us on $\mathcal E$ is given by the supremum norm: for a function $f \in C(\R^+,\R)$, we write $\|f\|=\sup_{s \ge 0} |f(s)|$.
Let $e^{(\sigma)}=(e^{(\sigma)}(s), 0 \leq s \leq \sigma)$ be a Brownian excursion of length $\sigma>0$.  We omit the superscript in the case of a standard Brownian excursion $(e(s),  0\le s\le 1)$. Note that, by Brownian scaling, we have
\[
(e^{(\sigma)}, 0\le s\le \sigma)\equidist (\sqrt \sigma \cdot e(s/\sigma), 0\le s\le \sigma).
\]
Then the \emph{area under the excursion} $e^{(\sigma)}$ is
\[
A(\sigma) := \int_0^\sigma e^{(\sigma)}(s) ds \equidist \sigma^{3/2} \int_0^1 e(s) ds.
\] 
The random variable $A(2)$ has the so-called \emph{Airy distribution}. This distribution has a rather complicated form but, for our purposes, it will suffice to note that its Laplace transform, $\phi: \mathbb{C} \to \mathbb{C}$ given by  $\phi(z) = \Ec{\exp(-z \int_0^1 e(x) dx )}$, is an entire function (see Janson \cite{janson07excursionarea} for details); in particular, it is finite for $z = -1$. 
For $\sigma>0$, we define the \emph{tilted excursion} of length $\sigma$, $\tilde{e}^{(\sigma)} = (\tilde{e}^{(\sigma)}(s), 0 \leq s \leq \sigma)\in \mathcal E$, to be an excursion whose distribution is characterized by
\begin{equation} \label{eqn:changemeas}
\Prob{\tilde{e}^{(\sigma)} \in \mathcal{B}} = \frac{\E{\I{e^{(\sigma)} \in \mathcal{B}} \exp \left( \int_0^{\sigma} e^{(\sigma)}(s) ds\right)}}{\E{\exp\left(\int_0^{\sigma} e^{(\sigma)}(s) ds\right)}},
\end{equation}
for $\mathcal{B} \subset \mathcal{E}$ a Borel set.  Here, the Borel sigma-algebra on $\mathcal{E}$ is that generated by open sets in the supremum norm $\|\cdot\|$. (Equation (\ref{eqn:changemeas}) gives a well-defined distribution since $\phi(-1) < \infty$.) As with the standard Brownian excursion, we will omit the superscript whenever the length of the tilted excursion is 1.
As previously, let $\mathbb{D}([0,\sigma], \R^+)$ be the space of non-negative c\`adl\`ag paths on $[0,\sigma]$, equipped with the Skorohod topology.

\begin{thm}\label{thm:conv_marked_tildedX}
Suppose that $p=p(m)$ is such that $m p^{2/3} \to \sigma$ as $m\to\infty$. Then, as $m\to\infty$, 
\[
((m/\sigma)^{-1/2} \tilde{X}^m(\fl{(m/\sigma)t}), 0 \leq t \leq \sigma) \convdist (\tilde{e}^{(\sigma)}(t), 0\le t\le \sigma),
\]
in $\mathbb D([0,\sigma], \R^+)$.
\end{thm} 

The proof consists in transferring known limits for uniform random trees over to tilted trees. 
We must first ensure that the change of measure defined by $\pc{\tilde T_m^p=T}\propto (1-p)^{-a(T)}$ is well-behaved when $p=O(m^{-3/2})$. 
To do so, we shall in fact first derive tail bounds on the maximum of the depth-first walk. (\citet{mar09} have proved similar bounds for the maxima of Dyck paths, which are essentially the contour processes of Catalan trees.) 
Let $T_m$ be a uniformly random tree on $[m]$, and let $X^m$ be the associated depth-first walk. 
\begin{lem}\label{lem:heightbound}
There exist constants $C\ge 0$ and $\alpha>0$ such that for all $m\in \Z^+$ and all $x \geq 0$, 
\[
\p{\|X^m\| \geq x \sqrt m} \leq C e^{-\alpha x^2}.
\]
\end{lem}
\begin{proof} To prove the lemma we use a connection with a queueing process which is essentially due to Borel \cite{borel1942etb}. 
Consider a queue with Poisson rate 1 arrivals and constant service time, started at time zero with a single customer in the queue. We may form a rooted tree (rooted at 
the first customer) associated with the queue process, run until the first time there are no customers in the queue, in the following manner: if a new customer joins the queue at time 
$t$, he is joined to the customer being served at 
time $t$. We denote the resulting rooted tree by $\mathscr{T}$. Then $\mathscr{T}$ is distributed as a Poisson(1) Galton--Watson tree and, hence, conditional on its size being $m$, as $T_m$ (viewed as an unlabeled tree) 
\cite{borel1942etb}. 
Viewing the arrivals as given by a Poisson process $\scr{Q}$ of intensity $1$ on $\mathbb{R}^+$, we may also associate an interpolated random walk to the process, by 
$S_t=|\scr{Q} \cap [0,t)|-t$, for $t \in \mathbb{R}^+$. Then $|\mathscr{T}|$ is precisely the first time $t$ that $S_t = -1$, i.e., that $|\scr{Q} \cap [0,t)|=t-1$. 
Furthermore, $\{S_t,t=1,2,\ldots,|\mathscr{T}|-1\}$
is distributed precisely as the depth-first 
walk of $\mathscr{T}$. (It is not {\em equal} to the depth-first walk, but to the breadth-first walk discussed in Section \ref{sec:random_graph_limit}.) 

Using the above facts, we may thus generate $\mathscr{T}$ conditional upon $|\mathscr{T}|=m$ as follows. First let $U_1,\ldots,U_{m-1}$ be independent and uniformly distributed on $[0,m]$, and let 
$\scr{U}=\{U_1,\ldots,U_{m-1}\}$, so that $\scr{U}$ is distributed as $|\scr{Q} \cap [0,m)|$ conditional on $|\scr{Q} \cap [0,m)|=m-1$. Next, let $\mu \in \{1,\ldots,m\}$ minimize $|\scr{U} \cap [0,\mu]|-\mu$, and 
apply a {\em cyclic rotation by $-\mu$} to all the points in $\scr{U}$ to obtain $\scr{U}'$. In other words, let 
\[
\scr{U}' = \{(U_1-\mu) \mod m,\ldots,(U_{m-1}-\mu) \mod m\}.
\]  
Then $m$ is precisely the first time $t$ that $|\scr{U}'\cap[0,t)| =t-1$, and $\scr{U'}$ is distributed precisely as $|\scr{Q} \cap [0,m)|$ conditional on $|\mathscr{T}|=m$ (see \cite{dwass69progeny}; this type of 
``rotation argument'' was introduced in \cite{andersen53fluctuations}). 

Now write $X^m$ for the depth-first walk of $T_m$, and let $X_i = |\scr{U'}\cap[0,i)|-i$ for $i=1,\ldots,m$. By the above, $\{X^m_1,\ldots,X^m_{m-1},-1\}$ and $\{X_1,\ldots,X_m\}$ are identically distributed.
In particular, 
\begin{align*} 
\|X^m\| 	& \eqdist \max_{0 \leq i \leq m} X_i \\
		& = \max_{0 \leq i \leq m} (|\scr{U} \cap[0,i)|-i) - \min_{0 \leq i \leq m} (|\scr{U} \cap[0,i)|-i) - 1 \\ 
		& = \max_{0 \leq i \leq m} (|\scr{U} \cap[0,i)|-i) + \max_{0 \leq i \leq m} (|\scr{U} \cap[i,m)|-(m-i)) \\
		& \leq \sup_{0 \leq t \leq m} (|\scr{U} \cap[0,t)|-t) + \sup_{0 \leq t \leq m} (|\scr{U} \cap[t,m)|-(m-t)).
\end{align*} 
The two suprema in the latter equation are identically distributed, and so 
\begin{equation}\label{identdist}
\p{\|X^m\| \geq 2x} 	\leq 2\p{\sup_{0 \leq t \leq m} (|\scr{U} \cap[0,t)|-t) \geq x}.
\end{equation}
For any fixed $t \in [0,m]$, let $P_t$ be the event that there is a point of $\scr{U}$ at $t$. 
For fixed $x \geq 0$ and $t \in [0,m]$, 
$E_{t,x}$ be the event that $|\scr{U}\cap[0,t]| = \lceil t+x\rceil$ but that $|\scr{U}\cap[0,s)| < s+x$ for all 
$0 \leq s \leq t$ (so in particular, there is a point at $t$). 
We then have 
\begin{align*}
\probC{E_{t,x}}{|\scr{U}\cap[0,t]|=\lceil t+x\rceil,P_t} ~=~ \probC{|\scr{U} \cap [0,s)| < s+x~\forall\ 0 \leq s < t}{|\scr{U}\cap [0,t)|=\lceil t+x\rceil-1}.
\end{align*}
Applying the ballot theorem for stochastic processes to the latter probability (see, e.g., \cite{takacs65comb} or p.~218 of \cite{kallenberg03foundations}), we obtain the bound 
\begin{equation}\label{ballot}
\probC{E_{t,x}}{|\scr{U}\cap[0,t]| = \lceil t+x\rceil,P_t} \leq 1- \frac{\lceil t+x\rceil-1}{t+x} < \frac{1}{t}.
\end{equation}
Furthermore, in order for $\{\sup_{0 \leq t \leq m} (|\scr{U} \cap[0,t)|-t) \geq x\}$ to occur, 
$E_{t,x}$ must occur for some $0 \leq t \leq m$. 
Since an infinitesimal interval $[t,t+dt]$ contains a point of $\scr{U}$ with probability $dt(m-1)/m$, and 
for each $t$, $|\scr{U} \cap[0,t)|$ is distributed as Bin$(m-1,t/m)$, it follows from (\ref{identdist}) and (\ref{ballot}) that
\begin{align*} 
\p{\|X^m\| \geq 2x} 	&\leq 2\int_0^m \p{E_{t,x}} dt \\
				&< 2\int_{0}^m \frac{1}{t} \p{|\scr{U} \cap[0,t)|=\lceil t+x\rceil} \frac{(m-1)}{m} dt\\ 
				&\leq 2 \int_{0}^m \frac{1}{t} e^{-x^2/(2(t+x/3))}dt,
\end{align*} 
where the last inequality follows from Chernoff's bound for Binomial random variables (see, for example, \cite{janson00random}). 
The conclusion follows easily for $x \leq m/2$, and thus for all $x$ (since we always have $\|X^m\|< m$). 
\end{proof} 

\begin{lem}\label{lem:conv_moments_area}
There exist universal constants $K, \kappa > 0$ such that the following holds. Fix $c > 0$ and $\xi > 0$, and suppose that $p \in (0,1)$ and $m \in \Z^+$ are such that $p \leq cm^{-3/2}$. Let 
$T_m$ be a uniform random tree on $[m]$.  Then
\[
\Ec{(1-p)^{-\xi a(T_m)}}< K e^{\kappa c^2\xi^2}.
\]
\end{lem}

\begin{proof}Fix $c$ and $\xi$ as above, and let $\lambda=2c\xi$. We may clearly restrict our attention to $m$ sufficiently large that $p\le c m^{-3/2} \leq 1/2$. 
For all such $m$ we have  
$$
(1-p)^{-\xi a(T_m)}\le e^{2\xi p a(T_m)} \le e^{\lambda m^{-3/2} a(T_m)}, 
$$ 
so it suffices to prove that $\sup_{m\ge 1}\Ec{e^{\lambda m^{-3/2} a(T_m)}} < K e^{\kappa \lambda^2/4}$ for some universal constants $K$ and $\kappa$. 
But $a(T_m)\le m \|X^m\|$, and so 
$$
\pc{m^{-3/2} a(T_m) \ge x} \le \pc{m^{-1/2}\|X^m\|\ge x}. 
$$ 
Since $\|X^m\|\le m$, it follows by Lemma~\ref{lem:heightbound} that
\begin{equation}\label{eq:tail_bound_area} 
\E{e^{\lambda m^{-3/2}a(T_m)}} \le \int_0^{m^{1/2}}e^{\lambda x}\pc{m^{-1/2}\|X^m\|\ge x} dx \le \int_0^{m^{1/2}}e^{\lambda x}\cdot Ce^{-\alpha x^2} dx.
\end{equation}
Completing the square in the last integrand so as to express the right-hand side as a Gaussian integral yields the claim with $\kappa=\alpha^{-2}$ and $K=C\sqrt{\pi/\alpha}$.
\end{proof}

\begin{proof}[Proof of Theorem~\ref{thm:conv_marked_tildedX}]
We assume $\sigma=1$ for notational simplicity; the general result follows by Brownian scaling. 
Again, let $(X^m(i), 0 \leq i \leq m)$ be the depth-first walk associated with a uniformly-chosen labeled tree. Its area is $a(T_m)=\sum_{i=0}^{m-1} X^m(i)$.  We know from Theorem \ref{mamothm} that 
\begin{equation} \label{eqn:vanilladfwconv}
(m^{-1/2} X^m(\fl{mt}), 0 \leq t \leq 1)  \convdist (e(t), 0 \leq t \leq 1).
\end{equation} 
We will henceforth want to think of $X^m$ as a function in $\mathbb{D}([0,1], \R^+)$ and will write $\bar{X}^m(s)$ instead of $m^{-1/2} X^m(\fl{m s})$.  For $h \in \mathbb{D}([0,1],\R^+)$, let $I(h) = \int_0^1 h(t) dt$; $I$ is a continuous functional of the path $h$.  It then follows from (\ref{eqn:vanilladfwconv}) that 
\[
m^{-3/2} a(T_m)=\frac{1}{m^{3/2}} \sum_{i=0}^{m-1} X^m(i) = \int_0^1 \bar{X}^m(t) dt \convdist \int_0^1 e(t) dt,
\]
jointly with the convergence in distribution of the depth-first walk.  

Now  suppose that $f: \mathbb{D}([0,1], \R^+) \to \R^+$ is any bounded continuous function.  Then 
\[
\E{f\left(m^{-1/2} \tilde{X}^m(\fl{m \, \cdot\,})\right)} 
= \frac{ \E{ f(\bar{X}^m) (1 - p)^{-m^{3/2} \int_0^1 \bar{X}^m(t) dt} } }
           { \E{ (1 - p)^{-m^{3/2} \int_0^1 \bar{X}^m(t) dt} } }
\]
and
\[
\E{f(\tilde{e})} = \frac{ \E{f(e) \exp \left( \int_0^1 e(x) dx \right) }}{\E{\exp \left( \int_0^1 e(x) dx \right)}}.
\]
Since $\int_0^1 \bar{X}^m(t)dt \convdist \int_0^1 e(t) dt$ and $p\sim m^{-3/2}$, we also have 
\[
(1 - p)^{-m^{3/2} \int_0^1 \bar{X}^m(t) dt} \convdist \exp \left(\int_0^1 e(t) dt \right).
\]
By Lemma~\ref{lem:conv_moments_area} the above sequence is uniformly integrable, and so  
we can deduce that $\Ec{f(m^{-1/2} \tilde{X}^m(\fl{m \,\cdot \,})}\to \Ec{f(\tilde e)}$, which implies that
$(m^{-1/2} \tilde{X}^m(\fl{m t}), 0 \leq t < 1) \convdist \tilde{e}$.
\end{proof}

Unsurprisingly, as in the case of uniform trees, we also have convergence of the height process of tilted trees towards tilted excursions.
 \begin{thm}\label{thm:conv_marked_tiltedh}Suppose that $p=p(m)$ is such that $mp^{2/3}\to \sigma$ as $m\to\infty$. Then, as $m\to\infty$, 
 $$
 ((m/\sigma)^{-1/2}\tilde H^m(\fl{(m/\sigma)s}), 0\le s\le \sigma) \convdist (2\tilde e^{(\sigma)}(s), 0\le s\le \sigma)
 $$ 
 in $\mathbb D([0,\sigma], \R^+)$.
 \end{thm} 
 The theorem follows straightforwardly from the following lemma and Lemma \ref{lem:conv_moments_area}, much as Theorem \ref{thm:conv_marked_tildedX} folllowed from Lemma \ref{lem:conv_moments_area}; its proof is omitted.

\begin{lem}\label{lem:dist_tilted_Xh}
Suppose that $p=p(m)$ is such that $m p^{2/3} \to \sigma$ as $m\to\infty$. 
For $m\ge 1$ let $\tilde T_m^p$ be a tree on $[m]$ sampled according to the distribution $\pc{\tilde T_m^p=T}\propto (1-p)^{-a(T)}$. Let $\tilde X^m$ and $\tilde H^m$ be the associated depth-first walk and height process. Then, there are constants $K$ and $m_0\ge 0$, such that for all $m\ge m_0$. 
$$
\pc{\|\tilde X^m-\tilde H^m/2\|\ge m^{3/8}} \le K m^{-1/16}.
$$
\end{lem}
\begin{proof}Let $T_m$ be a tree on $[m]$ chosen uniformly at random, and write $X^m$ and $H^m$ for its depth-first walk and height process. Then, by definition,
$$\pc{\|\tilde X^m - \tilde H^m/2\|\ge m^{3/8}} = \frac{\Ec{\I{\|X^m-H^m/2\|\ge m^{3/8}} (1-p)^{-a(T_m)}}}{\Ec{(1-p)^{-a(T_m)}}}.$$
Distinguishing between the trees with a ``large'' area, $a(T_m)\ge m^{25/16}$, and the others, we obtain
\begin{align*}
\pc{\|\tilde X^m-\tilde H^m/2\|\ge m^{3/8}} 
&\le \frac{\Ec{\I{a(T_m)\ge m^{25/16}}(1-p)^{-a(T_m)}}}{\Ec{(1-p)^{-a(T_m)}}} + \frac{\pc{\|X^m-H^m/2\|\ge m^{3/8}}(1-p)^{-m^{25/16}}}{\Ec{(1-p)^{-a(T_m)}}}\\
&\le \pc{a(\tilde T_m) \ge m^{25/16}} + \frac{e^{-\gamma m^{1/8}} e^{m^{1/16}}}{\Ec{(1-p)^{-a(T_m)}}},
\end{align*}
where the second inequality follows from the bound $\pc{\|X^m-H^m/2\|\ge m^{3/8}}\le e^{-\gamma m^{1/8}}$, for some $\gamma>0$ and all $m$ large enough, which is obtained from Theorem \ref{thm:MaMobound}. By Markov's inequality,
$$\pc{\|\tilde X^m-\tilde H^m/2\|\ge m^{3/8}} \le \Ec{m^{-3/2} a(\tilde T_m)}\cdot m^{-1/16} + \frac{e^{-\gamma (1+o(1))m^{1/8}}}{\Ec{(1-p)^{-a(T_m)}}}.$$
Finally, by Theorem~\ref{thm:conv_marked_tildedX} and Theorem~\ref{mamothm} together with Lemma~\ref{lem:conv_moments_area} we have
$$
\E{m^{-3/2}a(\tilde T_m)} \to \E{\int_0^{\sigma} \tilde e^{(\sigma)}(s) ds}<\infty \qquad \mbox{and}\qquad
\E{(1-p)^{-a(T_m)}}\to \E{\exp\pran{\int_0^{\sigma} e^{(\sigma)}(s) ds}}>0,
$$ 
as $m\to\infty$. 
\end{proof}



\section{The limit of connected components}
\label{sec:connected_components}

\subsection{Generating connected components of the random graph}

In this section, we discuss two different ways of constructing a connected graph from a tree.

\medskip
\noindent\textbf{Bijective encoding of connected graphs.} Consider a connected labeled graph $G$ on $m$ vertices, and recall that running the $\oDFS$ process on $G$ produces the depth-first tree $T(G)$. Recall that the edges permitted by $\oDFS$ are those between vertices which both lie in the stack $\mathcal O_i$ for some $i$.  By Lemma~\ref{permit}, $G$ can be recovered from $T=T(G)$ by adding some specific subset of the permitted edges.  

By Lemma~\ref{area}, we may consider permitted edges to be in bijective correspondence with the integral points $(i,j)$ lying above the $x$-axis and under the depth-first walk $X=(X(i), 1\le i< m)$ encoding $T$.  So suppose that we take the depth-first walk $X$ and place a mark at the point $(i,j)$ if there is an edge in $G$ between $v_i$, the $i$-th vertex explored in $\oDFS$ order (that is, the first vertex of $\mathcal O_i$), and the vertex lying in position $|\mathcal O_i|-j + 1$ in $\mathcal O_i$ (equivalently, the $j$-th vertex of $\mathcal O_i$ when counting from the end (or bottom) of the stack).
Call the resulting object a \emph{marked depth-first walk}.   Then clearly this gives a bijection between marked depth-first walks and connected graphs $G$.  

For a tree $T$ with depth-first walk $X$ and a pointset $\mathcal Q \subset \Z^+\times \Z^+$, let $G^X = G^X(T,\mathcal Q)$ be the graph obtained by adding to $T$ the edges corresponding to the points in $\mathcal Q \cap X$ where, for convenience, we define 
\begin{equation}\label{eq:def_intersection}
\mathcal S\cap f:= \{(x,y)\in \mathcal S: 0 < y \le f(x)\}\qquad \mbox{for all}~\mathcal S \subset \R^+\times \R^+~\mbox{and}~f:\R^+\to \R^+.
\end{equation}
A \emph{Binomial pointset of intensity $p$} is random subset of $\Z^+\times \Z^+$ in which each point is present independently with probability $p$.  The following lemma follows straightforwardly from Proposition~\ref{prop:dist_conn_comp}.

\begin{lem}\label{lem:unif_comp}Let $p=p\in (0,1)$. Let $\tilde T_m^p$ be a tree on $[m]$ sampled in such a way that $\Probc{\tilde T_m^p=T}\propto (1-p)^{-a(T)}$. Let $(\tilde X(i), 0\le i<m)$ be the associated depth-first walk. Let $\mathcal Q^p\subset \Z^+\times \Z^+$ be a Binomial pointset with intensity $p$, independent of $\tilde T_m^p$. Then, the graph $G^X(\tilde T_m^p, \mathcal Q^p)$ has the same distribution as $G_m^p$, a connected component of $G(n,p)$, $n\ge m$, conditioned to have $m$ vertices.
\end{lem}

It is convenient to devise a similar construction which uses the height process in place of the depth-first walk.  Although this second construction does not provide a bijection between graphs and ``marked height processes", it turns out to be closely related to the first construction, and useful to us in the asymptotics which follow.

\medskip
\noindent\textbf{Generating connected graphs from height processes.} Consider a tree $T$ on $[m]$ whose height process (here and below, by ``height process'' we always mean ``height process in $\oDFS$ order'') is given by $(H(i), 0\le i<m)$. Given a discrete pointset $\mathcal Q \subset \Z^+\times \Z^+$, we define a graph $G^H=G^H(T,\mathcal Q)$ on $[m]$ using the height process as follows. Take the tree $T$ and for $(i,j)\in \mathcal Q$ with $0 < 2j \le H(i)$, add an edge between the vertex $v_i$ and the vertex at distance $(2j-1)$ from the root $v_0$ on the unique path to $v_i$.  (As will be clarified below, the factor of two here is essentially the factor of two difference between the limits of the depth-first walk and the height process in the case of random trees; c.f.\ Theorem~\ref{mamothm}~and Lemma~\ref{lem:dist_tilted_Xh}.)  Note that upon adding these edges, we may produce a graph $G^H$ such that $T(G^H)\ne T$.  Moreover, the procedure $G^H(T,\,\cdot\,)$ does not provide a bijection, since it is not onto (one of the endpoints of every edge added is at odd distance from the root).

See Figure~\ref{fig:cycles} for an example of the two constructions.

\begin{figure}
	\centering
\begin{picture}(420,150)
\put(0,7){\includegraphics[scale=.5]{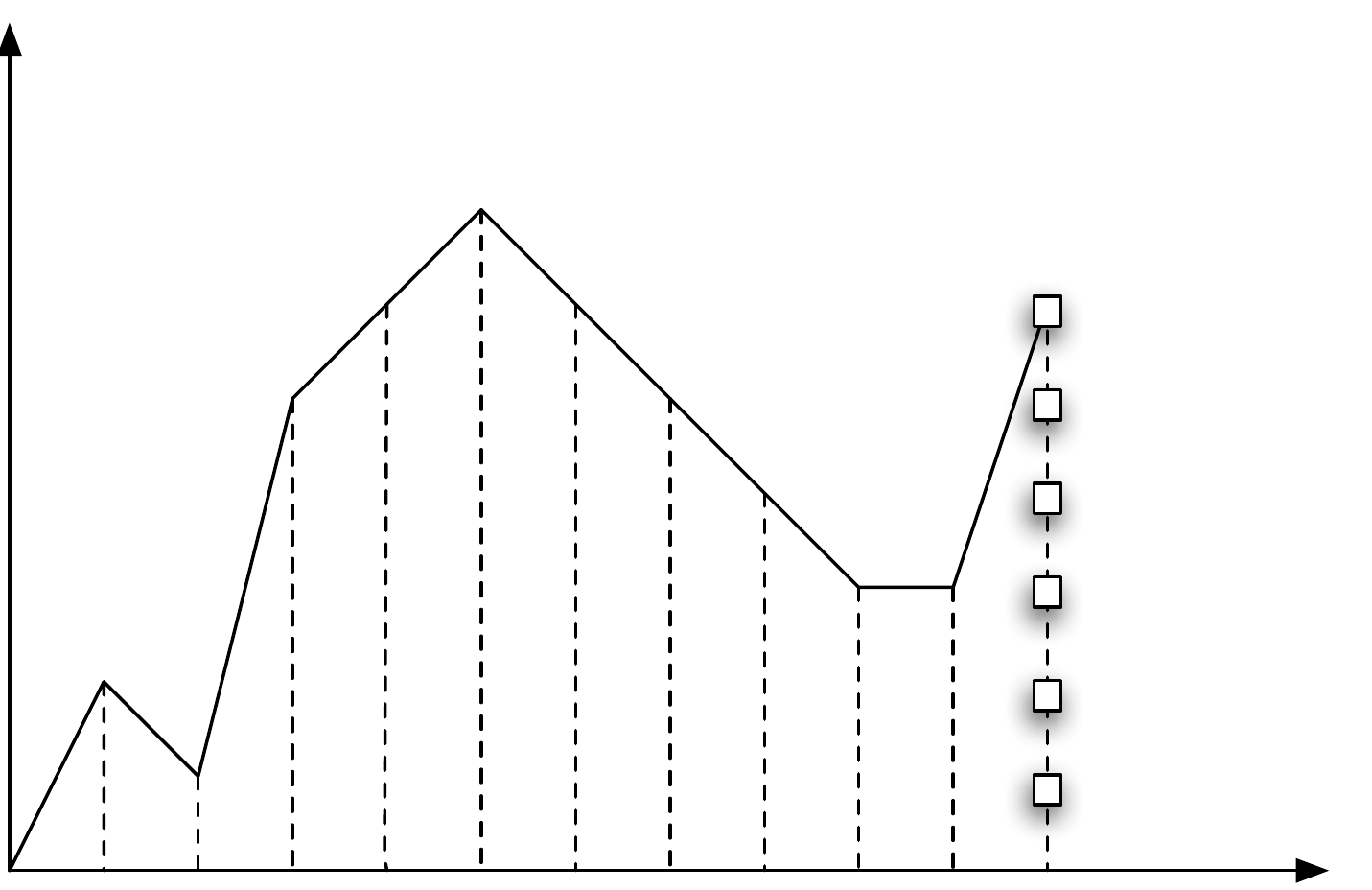}}
\put(155,-3){$i$}
\put(-5,140){$X(t)$}
\put(200,0){$t$}
\put(220,0){\includegraphics[scale=.5]{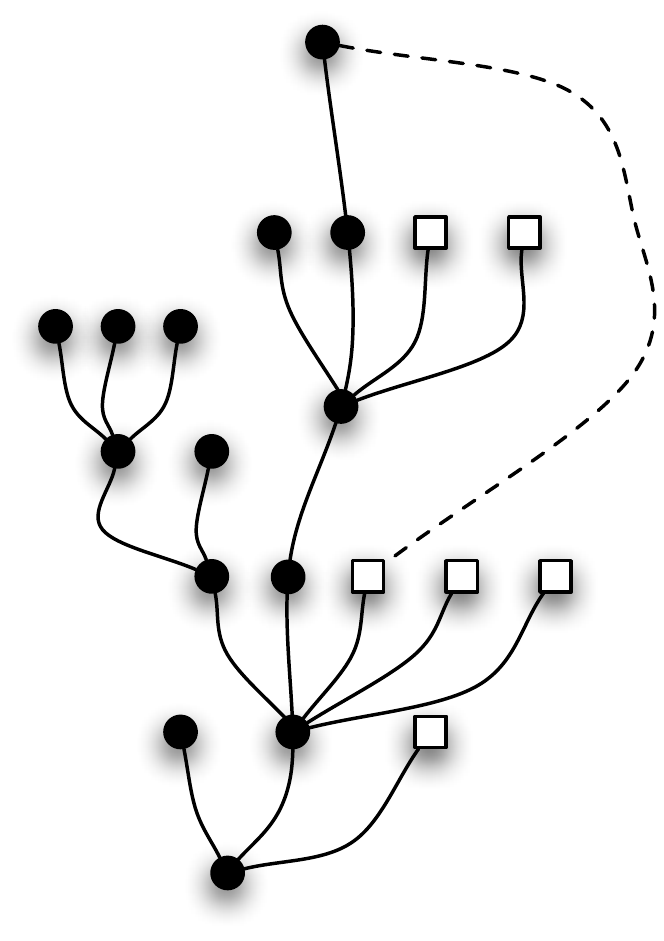}}
\put(230,40){$G^X$}
\put(350,0){\includegraphics[scale=.5]{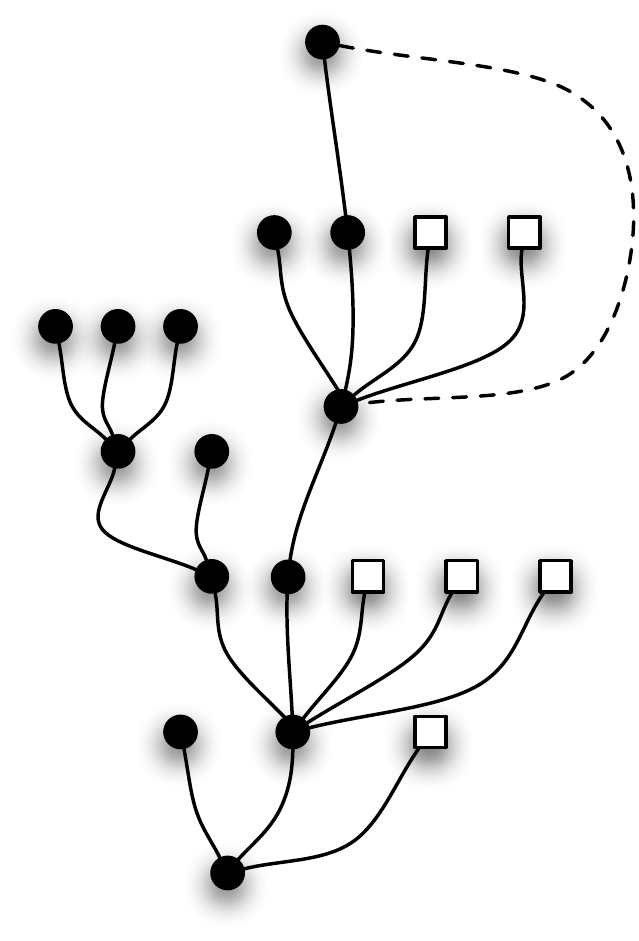}}
\put(163,20){$1$}
\put(163,35){$2$}
\put(163,49){$3$}
\put(163,63){$4$}
\put(163,77){$5$}
\put(163,91){$6$}
\put(290,28){$1$}
\put(270,60){$2$}
\put(285,60){$3$}
\put(300,60){$4$}
\put(280,109){$5$}
\put(295,109){$6$}
\put(360,40){$G^H$}
\put(256,130){$i$}
\put(386,130){$i$}
\end{picture}
\caption{\label{fig:cycles}Part of the depth-first walk of a tree $T$ is shown with the pointset $\mathcal Q = \{(i,2)\}$. In the walk picture, the boxes represent the nodes in $\mathcal O_i$. In the tree $T$, the black vertices have been fully explored. The two graphs $G^X(T,\mathcal Q)$ and $G^H(T,\mathcal Q)$ resulting from the correspondences with marked depth-first walk $X$ and height process $H$, respectively, are presented: on the left, $G^X(T,\mathcal Q)$ is obtained from $T$ by joining $v_i$ and the second vertex from the end of $\mathcal O_i$; on the right, $G^H(T,\mathcal Q)$ is obtained by joining $v_i$ and the vertex at distance 3 from the root on the to $v_i$ in $T$.}
\end{figure}
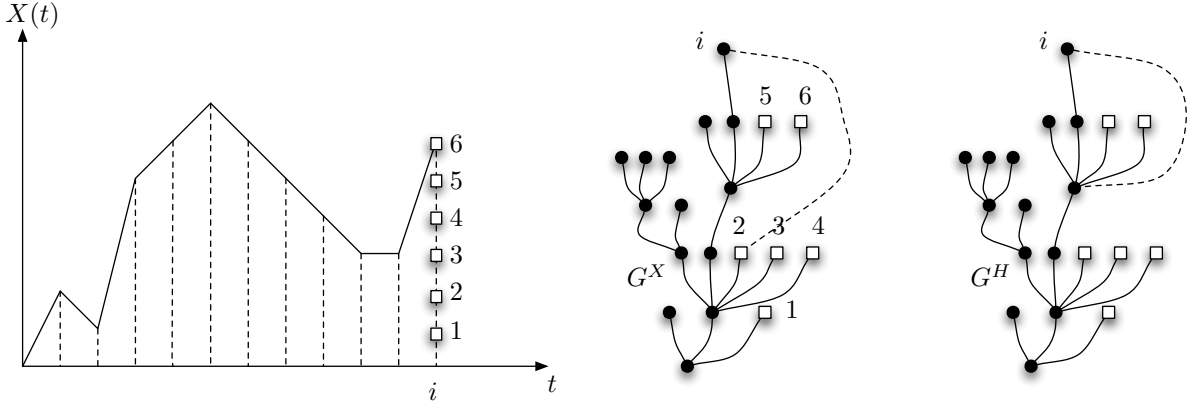

\medskip
The graphs $G^H(\tilde T_m^p,\mathcal Q^p)$ and $G^X(\tilde T^p_m, \mathcal Q^p)$ turn out to be very similar, and $G^H(\tilde T_m^p, \mathcal Q^p)$ is much easier to analyze than $G^X(\tilde T_m^p,\mathcal Q^p)$. In particular, the limiting object corresponding to $G^H(\tilde T_m^p, \mathcal Q^p)$ is very natural.  Write $\mathscr{L}$ for Lebesgue measure on $\R^+ \times \R^+$.

\begin{lem}\label{lem:conv_XP_exc_x} 
Let $p=p(m)$ be such that $m p^{2/3} \to \sigma$ as $m\to\infty$. Pick a labeled tree $\tilde T_m^p$ on $[m]$ in such a way that $\Probc{\tilde T_m^p=T}\propto (1-p)^{-a(T)}$ and let $\tilde{H}^m$ be the associated height process. Let $\mathcal Q^p\subset \Z^+\times \Z^+$ be a Binomial pointset of intensity $p$. Let $\mathcal P_m=\{((m/\sigma)^{-1} i, (m/\sigma)^{-1/2}j): (i,j)\in \mathcal Q^p\}$.  Then
\[
(((m/\sigma)^{-1/2} \tilde H^m(\fl{(m/\sigma)t}), 0\le t\le \sigma), \mathcal P_m \cap ((m/\sigma)^{-1/2} \tilde H^m(\fl{(m/\sigma)t})/2, 0\le t\le \sigma))\convdist (2\tilde e^{(\sigma)}, \mathcal P \cap \tilde e^{(\sigma)})
\]
as $n \to \infty$, where $\mathcal P$ is a homogeneous Poisson point process with intensity measure $\mathscr{L}$ on $\R^+ \times \R^+$, and $\mathcal P$ is independent of $\tilde{e}^{(\sigma)}$.
Convergence in the first co-ordinate is in $\mathbb{D}([0,\sigma],\mathbb{R}^+)$, and in the second co-ordinate is in the sense of the Hausdorff distance. 
\end{lem}

\begin{proof} We assume for notational simplicity that $\sigma=1$; the result for general $\sigma$ follows by Brownian scaling. (Note that the point process is rescaled independently of $\sigma$.)
Let $k \geq 1$ and let $A_1, A_2, \ldots, A_k \subset [0,1] \times \R^+$ be disjoint measurable sets.  Then for any $n \geq 1$ and any $1 \le i \le k$, the discrete counting function
\[
N_m(A_i)=\#\{(\fl{mx}, \fl{m^{1/2}y})\in \mathcal Q^p: (x,y)\in A_i\}
\]
is a Binomial random variable with parameters $\eta_{m}(A_i) = \#\{(\fl{m x}, \fl{m^{1/2} y}): (x,y) \in A_i, 0<\fl{m^{1/2}y}\le \fl{m x}\}$ and $p$.  Since 
\[
m^{-3/2} \eta_{m}(A_i) \to \mathscr{L}(A_i)
\]
and $m^{3/2} p \to 1$ as $m \to \infty$, $N_m(A_i) \convdist \mathrm{Poisson}(\mathscr{L}(A_i))$ for $1 \leq i \leq k$.  Moreover, the random variables $N_m(A_1), N_m(A_2), \ldots, N_m(A_k)$ are independent, since they count the points of $\mathcal Q^p$ in disjoint sets.  Thus $\mathcal P_m \to \mathcal P$ in distribution \cite{CoIs1980,kingman92poisson}.  

Suppose now that for each $m \geq 1$, $f_m : [0,1] \to \R^+$ is a continuous function, and that $f_m$ converges uniformly to some function $f: [0,1] \to \R^+$.  Then for any open set $A \subset [0,1] \times \R^+$, $\{(x,y) \in A: 0 <  y < f_m(x)\} \to \{(x,y) \in A: 0 < y < f(x)\}$ (in the sense of the Hausdorff distance).  It follows that $\mathcal P_m \cap f_m \to \mathcal P \cap f$ in distribution, since the Poisson process almost surely puts no points in the set $\{(x,y) \in A: y = f(x)\}$.

Now suppose that $g_m:\{0,1,\dots, m\}\to\Z^+$ and $(m^{-1/2}g_m(\fl{m t}), 0 \le t \le 1) \to (g(t), 0 \le t \le 1)$ in $\mathbb{D}([0,1],\R^+)$, where $g$ is continuous.  Then letting $\tilde{g}_m: [0,1] \to \R^+$ be the continuous interpolation of $(m^{-1/2} g_m(m t): t=0, m^{-1}, 2m^{-1},\dots, 1)$, we also have $\tilde{g}_m \to g$ uniformly.  Moreover, 
$$\mathcal{P}_n \cap (m^{-1/2}g_m(\fl{m t}), 0 \le t \le 1) = \mathcal{P}_m \cap \tilde{g}_m$$
since the functions agree at lattice points.  So we obtain that $\mathcal{P}_m \cap (m^{-1/2}g_m(\fl{m t}), 0 \le t \le 1) \convdist \mathcal{P} \cap g$.

Finally, since $\mathcal P_m$ and $\tilde{H}^m$ are independent, and $(m^{-1/2} \tilde H^m(\fl{m t}), 0\le t\le 1) \to (2\tilde e(t), 0\le t\le 1)$ in distribution by Theorem \ref{thm:conv_marked_tiltedh}, it follows easily that 
$$\mathcal P_m \cap (m^{-1/2}\tilde H^m(\fl{m t})/2, 0\le t\le 1) \convdist \mathcal P \cap \tilde e,$$ 
jointly with the convergence of the height process.
\end{proof}

\subsection{The limit object} \label{subsec:limob}

For $p=p(n)= 1/n+\lambda n^{-4/3}$, $\lambda \in \R$, any single one of the largest components of the random graph $G(n,p)$ has (random) size $m = m(n) \sim \sigma n^{2/3}$ \cite{Bollobas2001, janson00random, aldous97brownian}. Conditioned on having size $m$, such a connected component $G_m^p$ has the same distribution as $G^X(\tilde T_m^p, \mathcal Q^p)$.  As we have already remarked, Lemma~\ref{lem:dist_tilted_Xh} shows that $\tilde{X}^m$ and $\tilde{H}^m/2$ are close for large $m$.  Suppose now that we can prove that, for large $m$,  $G^X(\tilde T_m^p, \mathcal Q^p)$ and $G^H(\tilde T_m^p, \mathcal Q^p)$ are also ``close" in an appropriate sense.  Lemma~\ref{lem:conv_XP_exc_x} suggests that $G^H(\tilde T_m^p, \mathcal Q^p)$ should have a non-trivial limit in distribution when distances are rescaled by $n^{-1/3}$, which should, thus, also be the limit in distribution of $G_m^p$ similarly rescaled.  We will now define this limit object, $\mathcal{M}^{(\sigma)}$, by analogy with $G^H(\tilde T_m^p, \mathcal Q^p)$.

Let $\mathcal{T}$ be the real tree encoded by a height process $h$, as described in the introduction.  Recall that we think of $\mathcal{T}$ as a metric space, rooted at a vertex we call $\rho$.  Recall that $\tau$ denotes the canonical projection $\tau:[0,\sigma] \to \mathcal T$.  Let $\mathcal Q$ be a pointset in $\R^+\times \R^+$ such that there are only finitely many points in any compact set.  To each point $\xi = (\xi^x,\xi^y) \in \mathcal Q \cap (h/2)$, there corresponds a (unique) vertex $\tau(\xi^x) \in \mathcal{T}$ of height $h(\xi^x)$.  Now let $\hat{\tau}(\xi^x,\xi^y)$ be the vertex at distance $2 \xi^y$ from $\rho$ on the path $[[\rho,\tau(\xi^x)]]$.  Define a new ``glued'' metric space $g(h, \mathcal{Q})$ by identifying the vertices $\tau(\xi^x)$ and $\hat{\tau}(\xi^x,\xi^y)$ in $\mathcal{T}$, for each point $\xi \in \mathcal Q \cap (h/2)$, and taking the obvious induced metric.

Now let $\mathcal{P}$ be a Poisson point process with intensity measure $\mathscr{L}$ on $\R^+\times \R^+$, independent of $\tilde{e}^{(\sigma)}$, a tilted excursion of length $\sigma$.  Note that $\mathcal{P}$ almost surely only has finitely many points in any compact set.  Then we define the random metric space $\mathcal{M}^{(\sigma)} = g(2 \tilde{e}^{(\sigma)}, \mathcal{P})$ and write $\mathcal M = \mathcal M^{(1)}$.

In Theorem~\ref{thm:gh_connectedx} below, we will see that $\mathcal{M}^{(\sigma)}$ is indeed the scaling limit of a connected component of $G(n,p)$ conditioned to have size $m\sim \sigma n^{2/3}$.  In order to do this, we show that the metric spaces corresponding to $G^X(\tilde T_m^p, \mathcal Q^p)$ and $G^H(\tilde T_m^p, \mathcal Q^p)$ are very likely close in the Gromov--Hausdorff distance for large $m$, and that the metric space corresponding to $G^H(\tilde T_m^p, \mathcal Q^p)$ and $\mathcal{M}^{(\sigma)}$ are also close in the Gromov--Hausdorff distance for large $m\sim \sigma n^{2/3}$. 

The scaling limit of the whole of $G(n,p)$ is then be a collection of such continuous random components with random sizes; the components of size $o(n^{2/3})$ rescale to trivial continuous limits.
The proof of this is dealt with in Section~\ref{sec:random_graph_limit}. 

It is possible to give a more intuitively appealing description of $\mathcal M^{(\sigma)}$.  Take a continuum random tree with tilted distribution and independently select a random number $P$ of its leaves, each picked with density proportional to its height.  For each of these leaves there is a unique path to the root of the tree.  Independently for each of the selected leaves, pick a point uniformly along the path and identify the leaf and the selected point. Before we move on, we justify this description. 

\begin{prop} \label{lem:factsaboutlimit}
Using the same notation as above, the following statements hold.
\begin{compactenum}[1.]
\item Given $\tilde{e}^{(\sigma)}$, encoding the real tree $\tilde {\mathcal T}$, for any $(\xi^x,\xi^y) \in \mathcal P \cap \tilde e^{(\sigma)}$, $\xi^x$ has density
\[
\frac{\tilde{e}^{(\sigma)}(u)}{\int_0^\sigma \tilde{e}^{(\sigma)}(s) ds}
\]
on $[0,\sigma]$.  Moreover, the vertex $\tau(\xi^x)$ is almost surely a leaf of $\tilde{\mathcal{T}}$.
\item We have $d_{\tilde{\mathcal{T}}}(\rho, \hat{\tau}(\xi^x,\xi^y)) : = 2 \xi^y \equidist U d_{\tilde{\mathcal{T}}}(\rho, \tau(\xi^x))$, where $U$ is a uniform random variable on $[0,1]$, independent of $\xi^x$ and $\tilde{e}^{(\sigma)}$.
\item Given $\tilde{e}^{(\sigma)}$, the number of vertex identifications in $\mathcal{M}^{(\sigma)}$ has a Poisson distribution with mean $\int_0^\sigma \tilde{e}^{(\sigma)}(u) du$.
\end{compactenum}
\end{prop}

\begin{proof}
It is useful to keep in mind that there are two (independent) sources of randomness: one which gives rise to the excursion $\tilde{e}^{(\sigma)}$ (and hence the tree $\tilde{\mathcal{T}}$), and another which gives rise to the point process $\mathcal{P}$.

Given $\tilde{e}^{(\sigma)}$ and that $(\xi^x,\xi^y)$ is a point of $\mathcal P \cap \tilde e^{(\sigma)}$, it is straightforward to see that $\xi^x$ has claimed density on $[0,\sigma]$.  Hence, 
\[
\Cprobc{\tau(\xi^x) \in \mathcal{L}(\tilde{\mathcal{T}})}{\tilde{\mathcal{T}}} = \int_0^\sigma \I{\tau(u) \in \mathcal{L}(\tilde{\mathcal{T}})} \frac{\tilde{e}^{(\sigma)}(u)}{\int_0^x \tilde{e}^{(\sigma)}(s) ds}du.
\]
Recall that $\mu$ is the natural measure induced on $\tilde{\mathcal{T}}$ from Lebesgue measure on $[0,\sigma]$.  Since the distribution of $\tilde{e}^{(\sigma)}$ is absolutely continuous with respect to that of $e^{(\sigma)}$, a Brownian excursion of length $\sigma$, and $\mu(\mathcal L (\mathcal T))=\sigma$ \cite[][p. 60]{aldous91crt2}, we must also have $\mu(\mathcal{L}(\tilde{\mathcal{T}})) = \sigma$.  It follows that
\[
\Cprobc{\tau(\xi^x) \in \mathcal{L}(\tilde{\mathcal{T}})}{ \tilde{\mathcal{T}}} = 1
\]
for almost all $\tilde T$.  Hence, $\tau(\xi^x)$ is almost surely a leaf. 

Again using the fact that  $(\xi^x,\xi^y)$ is a point of $\mathcal P \cap \tilde{e}^{(\sigma)}$, given $\tilde{e}^{(\sigma)}$ and $\xi^x$, $\xi^y$ is uniformly distributed in $[0,e^{(\sigma)}(\xi^x)]$.  The second statement follows since $d_{\tilde {\mathcal T}}(\rho,\tau(\xi^x)) = 2 \tilde{e}^{(\sigma)}(\xi^x)$.  The third statement is immediate from the fact that the number of points in $\mathcal P \cap \tilde{e}^{(\sigma)}$ has a Poisson distribution with mean $\int_0^\sigma \tilde{e}^{(\sigma)}(u) du$.
\end{proof}

\subsection{Bounding the dissimilarity between $G^X$ and $G^H$}\label{sec:graphs_Xh}
In this section, we provide \emph{deterministic} bounds on the distances
\[
d_{GH}(G^X(T_1, \mathcal Q_1), G^H(T_1, \mathcal Q_1)) \qquad \mbox{and}\qquad d_{GH}(g(h_1, \mathcal Q_1), g(h_2, \mathcal Q_2))
\]
for fixed labeled trees $T_1$ and $T_2$, fixed height processes $h_1$ and $h_2$ and pointsets $\mathcal Q_1$ and $\mathcal Q_2$. These bounds are used in Section~\ref{sec:gh_connect} to obtain the scaling limit of connected components $G_m^p$. 

Consider the depth-first walk $(X(i), 0\le i<m)$ and the height process $(H(i), 0\le i<m)$ of a (deterministic) tree $T$ on $[m]$, together with a finite discrete point set $\mathcal Q \subset \Z^+ \times \Z^+$.  The two excursions and the pointset give rise to two graphs $G^X:=G^X(T,\mathcal Q)$ and $G^H:=G^H(T,\mathcal Q)$.  We will, for the moment, view $G^X$ and $G^H$ as metric spaces on $\{0,1,\ldots,m-1\}$.  We do this in the following way: first relabel the vertices of $T$ with $\{0,1,\ldots, m-1\}$ in depth-first order.  Then, in each case, add the relevant surplus edges and, finally, take the metrics to be the graph distances $d^X(\cdot,\cdot)$ and $d^H(\cdot,\cdot)$ induced by $G^X$ and $G^H$ respectively.  We will abuse notation by continuing to refer to these metric spaces as $G^X$ and $G^H$.  
Note that we have given an injection of the two graphs/metric spaces into a common set of points $V$ and so we can measure the distortion
\[
\sup_{x,y \in V} |d^H(x,y)-d^X(x,y)|
\] 
(with respect to this fixed injection) between them. This distortion provides an upper bound on the Gromov--Hausdorff distance (see, e.g., \cite{legall05survey}).

\begin{lem}\label{lem:gh_Xh_samepoints}Suppose that $\mathcal Q\cap X=\mathcal Q \cap (H/2)$, and write $k=|\mathcal Q \cap X|=|\mathcal Q \cap (H/2)|$. Then 
	$$d_{GH}(G^X, G^H)\le k (\|X-H/2\|+2).$$
\end{lem}

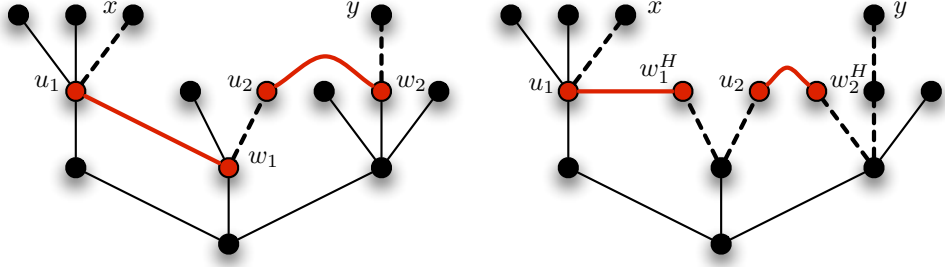
\begin{figure}
\centering
\begin{picture}(420,110)
\put(18,0){\includegraphics[scale=0.8]{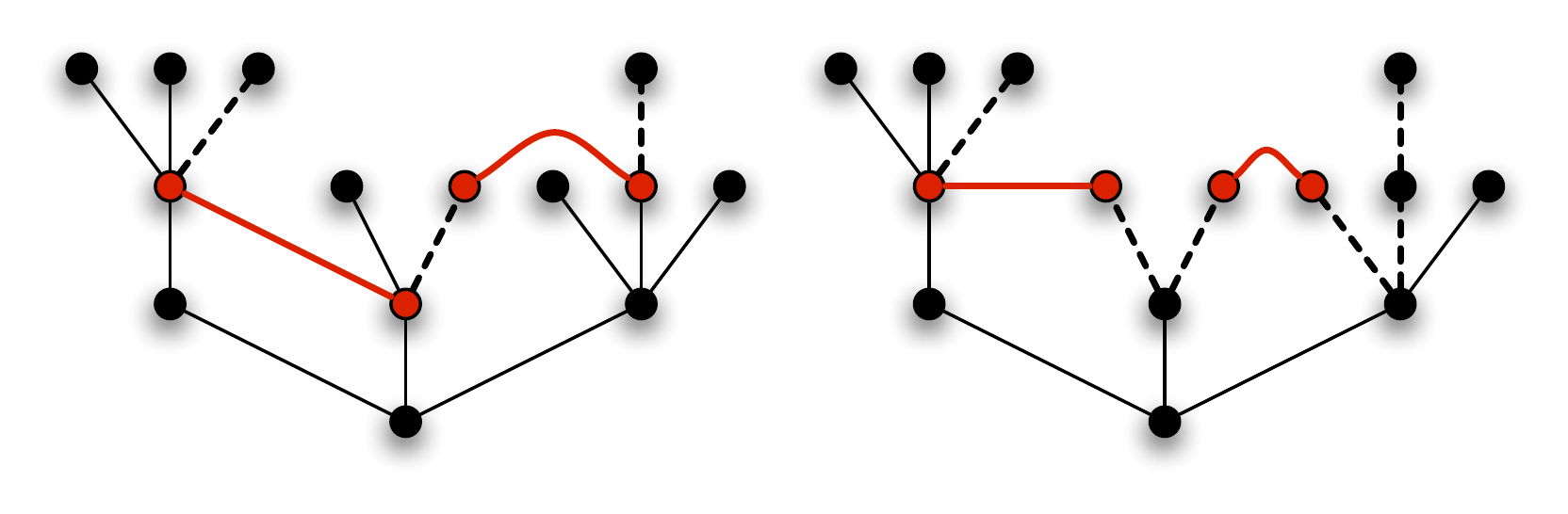}}
\put(70,110){$x$}
\put(162,110){$y$}
\put(44,82){$u_1$}
\put(117,82){$u_2$}
\put(181,82){$w_2$}
\put(231,81){$u_1$}
\put(273,86){$w_1^H$}
\put(125,53){$w_1$}
\put(303,82){$u_2$}
\put(345,82){$w_2^H$}
\put(276,110){$x$}
\put(369,110){$y$}
\end{picture}
\caption{\label{fig:shortcuts}Left: a graph $G^X$ and the shortest path between vertices $x,y$. Right: the corresponding path (which is no longer a shortest path) between $x$ and $y$ in the graph $G^H$. 
In both graphs, the shortcut edges are bold, and the remaining edges of the paths are dashed.}
\end{figure}

\begin{proof}
For the present proof, we will think of a path in a graph as an ordered list of vertices (which may contain repeats) which is such that there is an edge present between each pair of adjacent vertices.  We will write $\oplus$ for the operation of concatenation on ordered lists so that, for example, $(1,2,3) \oplus (2,4) = (1,2,3,2,4)$.  Fix two vertices $x$ and $y$, and consider a shortest path $\pi^X(x,y)$ between $x$ and $y$ in $G^X$. Let $(u_i, w_i)$, $0\le i \le \ell$ be the sequence of endpoints of the surplus edges in $\pi^X(x,y)$ in the order in which they appear when going from $x$ to $y$. Then we have 
\[
\pi^X(x,y)= \pi_T(x,u_1) \oplus \bigoplus_{i=1}^{\ell-1} \Big( (u_i,w_i) \oplus \pi_T(w_i,u_{i+1})\Big) \oplus (u_{\ell}, w_{\ell}) \oplus \pi_T(w_\ell, y),
\]
where $\pi_T(u,v)$ denotes the unique shortest path between $u$ and $v$ in the tree $T$. 
We will now define a path $\pi^H(x,y)$ in $G^H$ between the nodes $x$ and $y$. The path $\pi^H(x,y)$ will not necessarily be a shortest path from $x$ to $y$ in $G^H$, but it will go through the  vertices $u_i, w_i, 1\le i\le \ell$ in the order in which they appear in $\pi^X(x,y)$. (The relation between the paths $\pi^X(x,y)$ and $\pi^H(x,y)$ is depicted in Figure \ref{fig:shortcuts}.)  In particular, let 
\[
\pi^H(x,y)=\pi_T(x,u_1) \oplus \bigoplus_{i=1}^{\ell-1} \Big(\pi^{X,H}(u_i,w_i) \oplus \pi_T(w_i,u_{i+1}) \Big) \oplus \pi^{X,H}(u_{\ell},w_{\ell}) \oplus \pi_T(w_\ell, y),
\]
where $\pi^{X,H}(u_i,w_i)$, $1\le i\le \ell$, are paths that we will define in a moment.
We clearly have
\begin{equation}\label{eq:path_sum_edges}
|\pi^H(x,y)|-|\pi^X(x,y)| = \sum_{i=1}^\ell (|\pi^{X,H}(u_i,w_i)|-1).
\end{equation}
For $1\le i\le \ell$, $u_i$ and $w_i$ are joined in $G^X$ by an edge corresponding to a point $\xi_i\in \mathcal Q\cap X$. By assumption, $\xi_i\in \mathcal Q \cap (H/2)$ so $\xi_i$ also induces an edge $f_i$ in $G^H$ and, by definition, at least one of $u_i$ and $w_i$ must be an endpoint of $f_i$. Assume that $f_i$ joins $u_i$ to a vertex $w_i^H$ (an analogous argument works if $f_i$ is instead attached to $w_i$). Then we define $\pi^{X,H}(u_i,w_i)$ to be the path that consists of the edge $(u_i,w^H_i)$ followed by the shortest path $\pi_T(w_i^H, w_i)$ between $w_i^H$ and $w_i$ in $T$. Then, we have
\begin{equation}\label{eq:path_edge}|
	\pi^{X,H}(u_i,w_i)| -1 = |\pi_T(w_i^H, w_i)|.
\end{equation}
Now, by definition, $w_i^H$ lies on the path $\pi_T(\rho,u_i)$ in $T$ between $u_i$ and the root of $T$. Since when $u_i$ is the first element of the $\oDFS$ stack, the other elements are all at distance 1 from the path from the root to $u_i$, $w_i$ has a neighbor $\overleftarrow{w_i}$ on that same path (see Figure~\ref{fig:cycles}). It follows that $|\pi_T(w_i^H,w_i)| \le |H(\overleftarrow{w_i})-H(w_i^H)|+1.$
The edges $\{u_i,w_i\}$ in $G^X$ and $\{u_i,w_i^H\}$ in $G^H$ correspond to the same point $\xi_i\in \mathcal Q$, so $H(w_i^H)$ is also twice the index of $w_i$ (in reverse order) in the set $\mathcal O_i$. This number is exactly $2X(w_i)+2$ since the nodes in $\mathcal O_i$ which precede $w_i$ in $\oDFS$ order are exactly those that are still to explore when we arrive at $w_i$, and there are $X(w_i)$ of them.  Then,
\[
|\pi_T(w_i^H,w_i) |\le |H(\overleftarrow{w_i})-2X(w_i)|+3 \le |H(w_i)-2X(w_i)|+4 \le \|H-2X\|+4.
\]
Using this bound together with (\ref{eq:path_sum_edges}) and (\ref{eq:path_edge}), we finally obtain that
$$
d^H(x,y)-d^X(x,y)\le 2\ell (\|X - H/2\|+2) \le 2 k(\|X - H/2\| +2).
$$
A symmetric argument shows that $d^X(x,y)-d^H(x,y)\le 2k(\|X - H/2\|+2)$.  The claim then follows since the Gromov--Hausdorff distance is bounded by the distortion, i.e.,  $d_{GH}(G^X,G^H) \le \frac 1 2 \sup_{x,y} |d^X(x,y)-d^H(x,y)|$ (see, for example, \cite{legall05survey}).
\end{proof}

When the pointsets used to define two graphs are not identical, we measure their dissimilarity with the Hausdorff distance.  Let $\mathcal T_1$ and $\mathcal T_2$ be two real trees encoded by the height processes $(h_1(s), 0\le s\le \sigma)$ and $(h_2(s), 0\le s\le \sigma)$ respectively.  Let $\tau_1: [0,\sigma] \to \mathcal{T}_1$ and $\tau_2: [0,\sigma] \to \mathcal{T}_2$ be the canonical projections, as discussed in the introduction, and let $\rho_1 = \tau_1(0)$, $\rho_2 = \tau_2(0)$.  Write $d_{1}$ and $d_{2}$ for the pseudometrics on $[0,\sigma]$ given by
\begin{align*}
d_{1}(s,t) & = h_1(s) + h_1(t) - 2 \inf_{s \wedge t \le r \le s \vee t} h_1(r) \\
d_{2}(s,t) & = h_2(s) + h_2(t) - 2 \inf_{s \wedge t \le r \le s \vee t} h_2(r),
\end{align*}
and write $d_{\mathcal T_1}$ and $d_{\mathcal T_2}$ for the metrics induced on $\mathcal T_1$ and $\mathcal T_2$ respectively.  Suppose that $\mathcal Q_1$ and $\mathcal Q_2$ are two pointsets in $\R^+\times \R^+$.  Let $G_1=g(h_1, \mathcal Q_1)$ and $G_2=g(h_2, \mathcal Q_2)$.

\begin{lem}\label{lem:bound_gh_real}
Suppose that $k=|\mathcal Q_1 \cap (h_1/2)|=|\mathcal Q_2 \cap (h_2/2)|$ and $\delta = d_{H}(\mathcal Q_1\cap (h_1/2), \mathcal Q_2 \cap (h_2/2))$.  Then,
\[
d_{GH}(G_1, G_2) \le \frac{1}{2} (k+1) \pran{\delta + 12 \|h_1-h_2\|+4 \sup_{|r-r'|\le \delta}|h_2(r)-h_2(r')|}.
\]
\end{lem}

\begin{proof}The bound is again obtained by estimating the distortion induced by the natural correspondence $\mathcal R = \{(\tau_1(s),\tau_2(s)): 0\le s\le \sigma\}$ between the vertices of $\mathcal T_1$ and $\mathcal T_2$.  For $s,t \in [0,\sigma]$, let $\pi^1(s,t)$ be a shortest path between $\tau_1(s)$ and $\tau_1(t)$ in $G_1$ (we consider this path to be an ordered continuous set of vertices).  
The argument is now roughly the same as that of Lemma \ref{lem:gh_Xh_samepoints}: we construct a path $\pi^2(s,t)$ from $\tau_2(s)$ to $\tau_2(t)$ that is not too much longer than $\pi^1(s,t)$, by ``splitting $\pi^1(s,t)$ up at the shortcuts''. 

Parts of $\pi^1(s,t)$ follow geodesics in the real tree $\mathcal T_1$, and translating these to $\mathcal T_2$ poses no challenge not met in Lemma \ref{lem:gh_Xh_samepoints}.  
The difficulty arises when $\pi^1(s,t)$ uses shortcuts produced by the vertex identifications induced by $\mathcal Q_1\cap (h_1/2)$.  Of course, these shortcuts are not present in $G_2$, which has its own shortcuts.  
Recall that $|\mathcal Q_1\cap (h_1/2)| = |\mathcal Q_2\cap (h_2/2)|=k$, let $\xi_1,\xi_2,\ldots,\xi_k$ be the points of $\mathcal Q_1\cap (h_1/2)$, and let $\eta_1, \eta_2, \ldots, \eta_k$ be the points of $\mathcal Q_2 \cap (h_2/2)$.  Since $d_{H}(\mathcal Q_1\cap (h_1/2), \mathcal Q_2 \cap (h_2/2))= \delta$, by relabeling if necessary, we can arrange that $\sup_{1 \le i \le k} \|\xi_i - \eta_i\| \le \delta$. In building 
$\pi^2(s,t)$, we will then use the shortcut induced by $\eta_i$ whenever the shortcut induced by $\xi_i$ is used in $\pi^1(s,t)$. 

In order to formalize this strategy, we need to give a sense to the notion of the {\em direction} in which we traverse a shortcut. Let $[[u,v]]$ denote the geodesic from $u$ to $v$ in $\mathcal{T}_1$ or $\mathcal{T}_2$, where which tree we mean will be obvious from context.  Let $\oplus$ be the concatenation operator.  Suppose that $\pi^1(s,t)$ uses $\ell$ shortcuts, and assume without loss of generality that we have labeled the points of $\mathcal Q_1\cap (h_1/2)$ so that the shortcuts along the path $\pi^1(s,t)$ are induced by $\xi_i = (\xi_i^x, \xi_i^y)$ for $1 \le i \le \ell$.  In each such identification, one of the vertices is $\tau_1(\xi_i^x)$ and the other is $\hat{\tau}_1(\xi_i^x, \xi_i^y)$ (which is the vertex at distance $2 \xi_i^y$ from $\rho_1$ along the path $[[\rho_1,\tau_1(\xi_i^x)]]$).  For each $i\in\{1,\dots, \ell\}$, there exists (at least one) $\zeta_i\in [0,\sigma]$ such that $\tau_1(\zeta_i) = \hat{\tau}_1(\xi_i^x, \xi_i^y)$; pick one such arbitrarily.  

Now $\pi^1(s,\xi_i^x)$ is a sub-path of $\pi^1(s,t)$: it is precisely the portion of $\pi^1(s,t)$ connecting $\tau_1(s)$ to $\tau_1(\xi_i^x)$. Furthermore, the {\em pullback} of $\pi^1(s,\xi_i^x)$ to $\scr{T}_1$ will consist of 
$i$ geodesics plus a single isolated point, which will be either $\tau_1(\xi_i^x)$ or $\tau_1(\zeta_i)$. We say $\tau_1(\xi_i^x)$ is {\em followed  by} $\tau_1(\zeta_i)$ in $\pi^1(s,t)$ if 
$\tau_1(\zeta_i)$ is the isolated point, and otherwise say that $\tau_1(\zeta_i)$ is followed by $\tau_1(\xi_i^x)$ in $\pi^1(s,t)$.

If $\tau_1(\xi_i^x)$ is followed by $\tau_1(\zeta_i)$ in $\pi^1(s,t)$ then let $s_i = \xi_i^x$ and $t_i = \zeta_i$.  If, on the other hand, $\tau_1(\zeta_i)$ is followed by $\tau_1(\xi_i^x)$, let $s_i = \zeta_i$ and $t_i = \xi_i^x$.  Then
\[
\pi^1(s,t)= [[\tau_1(s),\tau_1(s_1)[[ \ \oplus \left(\bigoplus_{i=1}^{\ell-1} \ [[\tau_1(t_i), \tau_1(s_{i+1})[[ \right)  \oplus [[\tau_1(t_{\ell}),\tau_1(t)]].
\]

We now construct a path $\pi^2(s,t)$ between $\tau_2(s)$ and $\tau_2(t)$ in $G_2$. 
For each vertex identification in $G_2$, one of the vertices is $\tau_2(\eta_i^x)$; the other is $\hat{\tau}_2(\eta_i^x,\eta_i^y) = \tau_2(\gamma_i)$ for some $\gamma_i$ (where $\hat{\tau}_2(\eta_i^x, \eta_i^y)$ is the vertex at distance $2 \eta_i^y$ from the root on the path $[[\rho_2, \tau_2(\eta_i^x)]]$). If $s_i = \xi_i^x$ then let $\tilde{s}_i = \eta_i^x$ and $\tilde{t}_i = \gamma_i$; if $s_i = \zeta_i$ then let $\tilde{s}_i = \gamma_i$ and $\tilde{t}_i = \eta_i^x$.  Then the vertex identifications in $G_2$ are $(\tau_2(\tilde{s}_i), \tau_2(\tilde{t}_i))$.

Now set
\begin{align*}
\pi^2(s,t)  &= [[\tau_2(s),\tau_2(s_1)[[ \ \oplus \left(\bigoplus_{i=1}^{\ell-1} \ [[\tau_2(s_i), \tau_2(\tilde s_i)[[ \ \oplus \ [[\tau_2(\tilde t_i), \tau_2(t_i)[[ \ \oplus \ [[\tau_2(t_i), \tau_2(s_{i+1})[[ \right) \\
& \qquad \oplus \ [[\tau_2(s_{\ell}), \tau_2(\tilde s_{\ell})[[ \ \oplus \ [[\tau_2(\tilde t_{\ell}), \tau_2(t_{\ell})[[ \ \oplus \ [[\tau_2(t_{\ell}), \tau_2(t)]].
\end{align*}
The difference in length of $\pi^1(s,t)$ and $\pi^2(s,t)$ may be bounded as follows:
\begin{eqnarray}
|\pi^2(s,t)|-|\pi^1(s,t)| 
&\le&  |d_{1}(s,s_1) - d_{2}(s,s_1)| 
+ \sum_{i=1}^{\ell - 1} |d_{1}(t_i,s_{i+1}) - d_{2}(t_i,s_{i+1})|\nonumber \\
&&+ |d_{1}(t_{\ell},t) - d_{2}(t_{\ell},t)| 
 + \sum_{i=1}^{\ell} \left( d_{2}(s_i,\tilde s_i) + d_{2}(t_i,\tilde t_i) \right).
\label{eqn:length_bound}
\end{eqnarray}
We bound the first three terms on the right-hand side together.  Recalling the definition of the distances $d_1$ and $d_2$, we obtain that for any $r, r' \in [0,\sigma]$,
\[
|d_1(r,r')  - d_2(r,r') | \le 4 \|h_1 - h_2\|.
\]
Hence, the first three terms in (\ref{eqn:length_bound}) are together bounded above by $4 (\ell + 1) \|h_1 - h_2\|$.  We now turn to the last term in (\ref{eqn:length_bound}).  We have
\begin{equation} \label{eqn:a}
d_2(s_i, \tilde s_i) + d_2(t_i, \tilde t_i) = d_2(\xi_i^x,\eta_i^x) + d_2(\zeta_i,\gamma_i),
\end{equation}
for $1 \le i \le \ell$.  Since $\|\xi_i - \eta_i \| \le \delta$, we have $|\xi_i^x - \eta_i^x| \le \delta$ and so we get
\begin{equation} \label{eqn:b}
d_2(\xi_i^x, \eta_i^x) ~=~ h_2(\xi_i^x) + h_2(\eta_i^x) - 2 \inf_{\xi_i^x \wedge \eta_i^x \le r \le \xi_i^x \vee \eta_i^x} h_2(r) ~\le~ 2 \sup_{|r-r'| \le \delta} |h_2(r) - h_2(r')|.
\end{equation}
Now consider $d_2(\zeta_i,\gamma_i)$.  If $\hat{\tau}_2(\xi_i^x, \xi_i^y), \hat{\tau}_2(\eta_i^x, \eta_i^y) \in [[\tau_2(\xi_i^x), \tau_2(\eta_i^x)]]$ then we straightforwardly have
\begin{equation} \label{eqn:c}
d_2(\zeta_i,\gamma_i) \le d_2(\xi_i^x, \eta_i^x)  \le 2 \sup_{|r-r'| \le \delta} |h_2(r) - h_2(r')|.
\end{equation}
If, on the other hand, at least one of $\hat{\tau}_2(\xi_i^x, \xi_i^y), \hat{\tau}_2(\eta_i^x, \eta_i^y)$ is not in $[[\tau_2(\xi_i^x), \tau_2(\eta_i^x)]]$, then both must be on the path from the root $\rho_2$ to one of $\tau_2(\xi_i^x)$ or $\tau_2(\eta_i^x)$.  Then
\begin{equation} \label{eqn:d}
d_2(\zeta_i,\gamma_i) = |\xi_i^y - \eta_i^y| \le \delta.
\end{equation}
It follows from equations (\ref{eqn:a}) to (\ref{eqn:d}) that
\[
\sum_{i=1}^{\ell} \left( d_{2}(s_i,\tilde s_i) + d_{2}(t_i,\tilde t_i) \right) \le  \ell \left( \delta + 4 \sup_{|r-r'| \le \delta} |h_2(r) - h_2(r')| \right).
\]
Since $\ell\le k$, putting the different parts of (\ref{eqn:length_bound}) together, we get
\[
|\pi^2(s,t)|-|\pi^1(s,t)|\le 4 (k+1) \|h_1-h_2\| + k \pran{\delta + 4\sup_{|r-r'|\le \delta} |h_2(r)-h_2(r')|}.
\]
An identical argument provides a bound on $|\pi^1(s,t)|-|\pi^2(s,t)|$ and, since $\sup_{|r-r'|\le \delta}|h_1(r)-h_1(r')|\le 2 \|h_1-h_2\|+\sup_{|r-r'|\le \delta}|h_2(r)-h_2(r')|$, we obtain
$$
d_{GH}(G_1, G_2) \le \frac{1}{2}(k+1) \pran{\delta + 12 \|h_1 -h_2\| + 4 \sup_{|r-r'|\le \delta}|h_2(r)-h_2(r')|},
$$
which completes the proof.
\end{proof}

\subsection{Gromov--Hausdorff convergence of connected components}\label{sec:gh_connect}

We are now almost ready to prove the convergence of $G_m^p$, a connected component of $G(n,p)$ conditioned to have size $m$, to a continuum random graph.  For a suitable continuous embedding of $G_m^p$, we will prove the following theorem.

\begin{thm}\label{thm:gh_connectedx} Suppose that $\sigma>0$. Let $m=m(n)\in \Z^+$ be a sequence of integers such that $n^{-2/3}m \to \sigma$ as $n\to\infty$, and let $p=p(n)\in (0,1)$ be such that $pn\to 1$. Let $G_m^p$ be a connected component of $G(n,p)$ conditioned on having size $m$. Then, as $n\to\infty$,
\[
n^{-1/3}G_m^p \convdist \mathcal M^{(\sigma)},
\]
in the Gromov--Hausdorff distance.
\end{thm}

As mentioned before, the proof uses crucially the connection between $G^X(\tilde T_m^p, \mathcal Q^p)$, which has the same distribution as $G_m^p$, and $G^H(\tilde T_m^p, \mathcal Q^p)$, for which the limit object appears naturally.  

We now find it convenient to think of  $G^H(\tilde T_m^p, \mathcal Q^p)$ as a metric space with a continuous embedding.  This is done in the obvious way by viewing each edge of the graph as a line segment of length 1, with the natural distance.  Write $m^{-1/2} G^H(\tilde T_m^p, \mathcal Q^p)$ for this metric space with all edges rescaled by $m^{-1/2}$.  We will begin by proving the following lemma.

\begin{lem}\label{lem:sl_marked_height} Suppose that $\sigma>0$. Let $p=p(m)\in (0,1)$ be such that $m p^{2/3} \to \sigma$.  Then as $m \to \infty$,
\[
m^{-1/2}G^H(\tilde T_m^p, \mathcal Q^p) \convdist \mathcal M^{(\sigma)},
\]
in the Gromov--Hausdorff distance.
\end{lem}

\begin{proof} We again assume for simplicity that $\sigma=1$, the general case following by Brownian scaling. 
We wish to use Lemma~\ref{lem:bound_gh_real}.  However, somewhat inconveniently, $m^{-1/2}G^H(\tilde T_m^p, \mathcal Q^p)$ is not quite the same as the real graph $g(\tilde H^m_c, \mathcal Q^p)$, where $\tilde H^m_c$ is the continuous interpolation of $(m^{-1/2}H^m(mt): t = 0,m^{-1}, 2m^{-1}, \ldots, 1)$.  The reasons for this are twofold.  Firstly, the real tree which is the natural continuous embedding of $m^{-1/2} \tilde T_m^p$ is \emph{not} the real tree coded by $H^m_c$, but is rather the real tree coded by $(m^{-1/2} \tilde{C}^m(2mt), 0 \le t \le 1)$, where $\tilde{C}^m$ is the contour process of the tree $T_m^p$.  Secondly, in $g(\,\cdot\,,\,\cdot\,)$ the points of the pointset induce vertex identifications, rather than the insertion of an edge of length $m^{-1/2}$ between the vertices concerned.  Neither of these facts will pose a problem, but they make the line of reasoning somewhat less clear.

We will first deal with approximating $m^{-1/2}G^H(\tilde T_m^p, \mathcal Q^p)$ by $g(h, \mathcal Q)$ for a suitable function $h$ and a suitable pointset $\mathcal Q$.  Write $\tilde{C}^m_r$ for the rescaled version of $\tilde{C}^m$, i.e.,
\[
\tilde{C}^m_r(t) := m^{-1/2} \tilde{C}^m(2mt), \quad 0 \le t \le 1.
\]
Then this function encodes the correct real tree.  However, if we use the points $\mathcal{P}_m = \{(m^{-1} i, m^{-1/2}j): (i,j) \in \mathcal Q^p\}$, we identify the \emph{wrong} vertices.  In order to get the correct vertices, we define 
\[
K(i) = 2i - \tilde{H}^m(i), \quad 0 \le i \le m-1.
\]
Then we have that the contour process travels from the $(i+1)$-st vertex visited in depth-first order to the $(i+2)$-nd in the interval $[K(i), K(i+1)]$ (see Section 2.4 of \citet{duquesnelegall02randomtrees} for a proof).  In particular, $\tilde C^m(K(i)) = \tilde H^m(i)$.  Now let 
$$\mathcal Q_m = \{((2m)^{-1} K(i), m^{-1/2} j): (i,j) \in \mathcal Q^p, 0 < j \le \tilde H^m(i) \}.$$  Then $\mathcal Q_m$ gives the right set of vertex identifications to make.  So we will approximate $m^{-1/2}G^H(\tilde T_m^p, \mathcal Q^p)$ by $g(\tilde{C}^m_r, \mathcal Q_m)$.

We now need to deal with the fact that, although $g(\tilde{C}^m_r, \mathcal Q_m)$ has the correct underlying tree and identifies the correct vertices, we actually insert edges (of length $m^{-1/2}$) in $m^{-1/2}G^H(\tilde T_m^p, \mathcal Q^p)$.  Suppose there are $k$ vertex identifications.  Take the correspondence between the vertices of the metric spaces $g(\tilde{C}^m_r, \mathcal Q_m)$ and $m^{-1/2}G^H(\tilde T_m^p, \mathcal Q^p)$ given by the fact that they have the same underlying (real) tree.  Now, $m^{-1/2}G^H(\tilde T_m^p, \mathcal Q^p)$ has line segments inserted between vertices which are identified in $g(\tilde{C}^m_r, \mathcal Q_m)$.  In the correspondence, map all of the elements of one of these segments in $m^{-1/2}G^H(\tilde T_m^p, \mathcal Q^p)$ to the single (identified) vertex in $g(\tilde{C}^m_r, \mathcal Q_m)$.  Then the distortion of this correspondence is at most $k m^{-1/2}$, since the distance between any two vertices in $m^{-1/2}G^H(\tilde T_m^p, \mathcal Q^p)$ differs by at most $k m^{-1/2}$ from the distance between the corresponding vertices in $g(\tilde{C}^m_r, \mathcal Q_m)$.  It follows that
\begin{equation} \label{eqn:edges_vanish}
d_{GH}(m^{-1/2}G^H(\tilde T_m^p, \mathcal Q^p),g(\tilde{C}^m_r, \mathcal Q_m)) ~\le~ \frac{1}{2} m^{-1/2}|\mathcal P_m \cap (m^{-1/2} \tilde{H}^m(\fl{m \,\cdot\,}/2)|.
\end{equation}

Recall that $\mathcal M = g(2 \tilde{e}, \mathcal P)$.  We next bound $d_{GH}(g(\tilde{C}^m_r, \mathcal Q_m), g(2 \tilde{e}, \mathcal P))$.
In order to do this using Lemma~\ref{lem:bound_gh_real}, we need an upper bound on the Hausdorff distance between $\mathcal Q_m$ and $\mathcal P \cap \tilde{e}$.  We have
\begin{equation} \label{eqn:Hausdorffpointsets}
d_H(\mathcal Q_m, \mathcal P \cap \tilde{e}) ~\le~ d_H(\mathcal Q_m, \mathcal P_m \cap (m^{-1/2} \tilde{H}^m(\fl{m \,\cdot\,})/2)) + d_H(\mathcal P_m \cap (m^{-1/2} \tilde{H}^m(\fl{m \,\cdot\,})/2), \mathcal P \cap \tilde{e}).
\end{equation}
There is a one-to-one correspondence between the vertices of $\mathcal Q_m$ and those of $\mathcal P_m \cap (m^{-1/2} \tilde{H}^m(\fl{m \,\cdot\,})/2)$; indeed, by definition, each point $(m^{-1}i, m^{-1/2}j) \in \mathcal P_m \cap (m^{-1/2} \tilde{H}^m(\fl{m \,\cdot\,})/2)$ corresponds to a point $((2m)^{-1} K(i), m^{-1/2}j) \in \mathcal Q_m$.  But then
\begin{equation} \label{eqn:pointsmatchup}
d_H(\mathcal Q_m, \mathcal P_m \cap (m^{-1/2} \tilde{H}^m(\fl{m \,\cdot\,})/2)) 
\le \frac{1}{m} \sup_{0 \le i < m} \left| i - \frac{K(i)}{2} \right| =  \frac{1}{m} \sup_{0 \le i < m} | \tilde{H}^m(i) |.
\end{equation}

By Lemma~\ref{lem:conv_XP_exc_x}, Theorem \ref{mamothm} and Skorohod's representation theorem, there is a probability space where 
\[
((m^{-1/2} \tilde H^m(\fl{mt}), 0\le t\le 1), (\tilde C^m_r(t), 0 \le t \le 1), \mathcal P_m \cap (m^{-1/2} \tilde H^m(\fl{mt})/2, 0\le t\le 1)) 
\to (2\tilde e, 2\tilde e, \mathcal P \cap \tilde e)
\]
almost surely.  Let $\epsilon>0$. Then, there exists an almost surely finite random variable $Y$ such that for all $m\ge Y$, the following hold: 
\begin{itemize}
\item $|\mathcal P_m \cap (m^{-1/2} \tilde{H}^m(\fl{m \,\cdot\,})/2)|=|\mathcal P \cap \tilde e|$,
\item $d_{H}(\mathcal P_m \cap (m^{-1/2} \tilde{H}^m(\fl{m \,\cdot\,})/2)), \mathcal P \cap \tilde e) \le \epsilon$, 
\item $\sup_{0\le t \le 1}|\tilde C^m_r(t)- 2\tilde e(t)|\le \epsilon$, and
\item $\sup_{0 \le i < m} | \tilde{H}^m(i) | \le m\epsilon$.
\end{itemize}
Note that $\mathcal Q_m = \mathcal Q_m \cap (\tilde C^m_r/2)$ since the points contained in $\mathcal Q_m$ are already below  $(\tilde C^m_r/2)$.  Then, by Lemma~\ref{lem:bound_gh_real}, (\ref{eqn:Hausdorffpointsets}) and (\ref{eqn:pointsmatchup}), for all $m\ge Y$,
\[
d_{GH}(g(\tilde{C}^m_r, \mathcal Q_m), g(2 \tilde e, \mathcal P))
\le \frac{1}{2} (|\mathcal P \cap \tilde e| + 1) \left( 14 \epsilon + 8 \sup_{|s-t|\le 2 \epsilon}|\tilde e(s)-\tilde e(t)| \right).
\]
(We remark that Lemma \ref{lem:bound_gh_real} is applied with $\delta=2\epsilon$, not $\delta=\epsilon$. This is essentially because of the factor 2 rescaling of the point sets that occurs prior to vertex identification.)
By (\ref{eqn:edges_vanish}), we then get that
\begin{equation} \label{eqn:almostthere}
d_{GH}(m^{-1/2}G^H(\tilde T_m^p, \mathcal Q^p),  g(2 \tilde e, \mathcal P)) \le \frac{1}{2} ((1 + m^{-1/2})|\mathcal P \cap \tilde e| + 1) \left( 14 \epsilon + 8 \sup_{|s-t|\le 2 \epsilon}|\tilde e(s)-\tilde e(t)| \right)
\end{equation}
for all $m \ge Y$.

Now note that, by definition, the excursion measure associated with a tilted excursion $\tilde e$ is absolutely continuous with respect to It\^o's excursion measure. It follows by Levy's Modulus of Continuity Theorem (see, for example, \cite{RoWi1994,ReYo2004}) that a tilted excursion $\tilde e$ satisfies
\begin{equation}\label{eq:gh_as_prod}
\p{\sup_{|s-t|\le 2 \epsilon}|\tilde e(s)-\tilde e(t)|\ge \epsilon^{1/4}}  \to 0
\end{equation}
as $\epsilon \to 0$.
Furthermore, given $\tilde e$, $|\mathcal P \cap \tilde e|$ has a Poisson distribution with mean $\int_0^1 \tilde e(s)ds$, and so
\begin{align*}
\pc{|\mathcal P \cap \tilde e| \ge \epsilon^{-1/8}} 
&\le \probC{|\mathcal P \cap \tilde e|\ge \epsilon^{-1/8}}{\int_0^1 \tilde e(s)ds \le \epsilon^{-1/16}} + \p{\int_0^1 \tilde e(s)ds \ge \epsilon^{-1/16}}\\
&\le \epsilon^{1/16} + \E{\int_0^1 \tilde e(t)dt} \epsilon^{1/16},
\end{align*}
by Markov's inequality. Since $\int_0^1 \tilde e(t)dt$ has finite expectation and $Y<\infty$ almost surely, we obtain from (\ref{eqn:almostthere}) and (\ref{eq:gh_as_prod}) that 
\begin{equation} \label{eqn:firstbit}
d_{GH}(m^{-1/2}G^H(\tilde T_m^p, \mathcal Q^p),  g(2 \tilde e, \mathcal P)) \convprob 0.
\end{equation}
The result follows.
\end{proof}

\begin{proof}[Proof of Theorem~\ref{thm:gh_connectedx}]  Once again, we treat the case $\sigma=1$, which implies the general result. By Lemma~\ref{lem:unif_comp}, $G_m^p$ is distributed as $G^X(\tilde T_m^p, \mathcal Q^p)$. Recall from Lemma \ref{lem:conv_XP_exc_x} that $\tilde X^m$ and $\tilde H^m$ denote the depth-first walk and height process of $\tilde T_m^p$, respectively.  In order to apply the result of Lemma~\ref{lem:sl_marked_height}, it will be convenient to be working with almost sure convergence instead of weak convergence.  By Lemmas~\ref{lem:dist_tilted_Xh} and \ref{lem:conv_XP_exc_x} and Skorohod's representation theorem, there is a probability space in which $m^{-1/2}\|\tilde X^m-\tilde H^m/2\| \to 0$ almost surely and
\begin{align*}
& \Bigg(
\pran{\frac{\tilde H^m(\fl{mt})}{\sqrt{m}}, 0\le t\le 1}, \pran{\frac{\tilde C^m(2mt)}{\sqrt{m}}, 0 \le t \le 1}, \pran{\frac{\tilde{X}^m(\fl{mt})}{\sqrt{m}}, 0\le t \le 1}, \\
& \hspace{7.3cm} \mathcal P_m \cap \pran{\frac{\tilde H^m(\fl{mt})}{2\sqrt m}, 0\le t\le 1} \Bigg)
\to (2\tilde e, 2\tilde e, \tilde e, \mathcal P \cap \tilde e),
\end{align*}
almost surely.  Recall that $\mathcal P_m = \{(m^{-1}i, m^{-1/2}j): (i,j) \in \mathcal Q^p\}$.  Then the size of the symmetric difference
\[
(\mathcal Q^p \cap \tilde{X}^m) \triangle (\mathcal Q^p \cap (\tilde{H}^m/2))
\]
is stochastically dominated by a Binomial random variable with parameters $\ce{m \|\tilde{X}^m - \tilde{H}^m/2\|}$ and $m^{-3/2}$.  Since $m^{-1/2}\|\tilde X^m-\tilde H^m/2\| \to 0$,
\begin{equation} \label{eqn:symmdiff}
|(\mathcal Q^p \cap \tilde{X}^m) \triangle (\mathcal Q^p \cap (\tilde{H}^m/2))| \convprob 0.
\end{equation}
Since this random variable is integer-valued, it follows that $\Probc{\mathcal Q^p \cap \tilde{X}^m = \mathcal Q^p \cap (\tilde{H}^m/2)} \to 1$ as $m \to \infty$.
Then, by Lemma~\ref{lem:gh_Xh_samepoints}, on the event $\{\mathcal Q^p \cap \tilde{X}^m = \mathcal Q^p \cap (\tilde{H}^m/2)\}$ we have
\[
d_{GH}(G^X(\tilde T_m^p, \mathcal Q^p), G^H(\tilde T_m^p, \mathcal Q^p))
\le |\mathcal Q^p \cap \tilde{X}^m| \cdot (\| \tilde{X}^m - \tilde{H}^m/2 \| + 2).
\]
So, for any $\epsilon>0$,
\begin{align*}
& \pc{d_{GH}(G^X(\tilde T_m^p, \mathcal Q^p), G^H(\tilde T_m^p,\mathcal Q^p))\ge \epsilon m^{1/2}} \\
&\quad \le~\pc{\|\tilde X^m-\tilde H^m/2\|\ge \epsilon m^{3/8}-2} + \pc{|\mathcal Q^p\cap \tilde X^m|\ge m^{1/8}} 
+ \Probc{\mathcal Q^p \cap \tilde{X}^m \neq \mathcal Q^p \cap (\tilde{H}^m/2)}\\
& \quad \le~(1+o(1))K \epsilon^{-1/6} m^{-1/16} + m^{-1/8} \Ec{|\mathcal Q^p \cap \tilde X^m|} 
+ \Probc{\mathcal Q^p \cap \tilde{X}^m \neq \mathcal Q^p \cap (\tilde{H}^m/2)},
\end{align*}
by Lemma~\ref{lem:dist_tilted_Xh} and Markov's inequality. Since $\Ec{|\mathcal Q^p \cap \tilde X^m|}\to \E{|\mathcal P\cap \tilde e|}=\Ec{\int_0^1 \tilde e(s) ds}<\infty$, the right-hand side tends to zero.  Finally,
\[
d_{GH}(m^{-1/2} G^X(\tilde{T}_m^p, \mathcal Q^p), \mathcal M)
~\le~ m^{-1/2} d_{GH}(G^X(\tilde{T}_m^p, \mathcal Q^p), G^H(\tilde{T}_m^p, \mathcal Q^p))
+ d_{GH}(m^{-1/2} G^H(\tilde{T}_m^p, \mathcal Q^p), \mathcal M)
\]
and so from (\ref{eqn:firstbit}) we obtain that $d_{GH}(m^{-1/2} G^X(\tilde{T}_m^p, \mathcal Q^p), \mathcal M) \convprob 0$. Since $G^p_m$ is distributed as $G^X(\tilde{T}_m^p, \mathcal Q^p)$ 
and $m^{-1/2}n^{1/3} \to 1$, the result follows. 
\end{proof}

We finish this section by stating an easy corollary on the number of surplus edges of $G_m^p$.
 
 \begin{cor}
 Suppose that $m = m(n)$ is such that $n^{-2/3}m \to \sigma$ as $n \to \infty$ and let $p=p(n)$ be such that $pn \to 1$.  Then as $n \to \infty$,
 \[
 s(G_m^p) \convdist \mathrm{Poisson}\left( \int_0^{\sigma} \tilde e^{(\sigma)}(u) du \right).
 \]
 \end{cor}
 
 \begin{proof}
 This follows immediately from the observation that
 \[
 s(G_m^p) \equidist |\mathcal Q^p \cap \tilde{X}^m| = |\mathcal P_m \cap ((m/\sigma)^{-1/2} \tilde X^m(\fl{(m/\sigma) t})/2, 0\le t\le \sigma)|,
 \]
 and from Lemma~\ref{lem:dist_tilted_Xh}, (\ref{eqn:symmdiff}) and Proposition~\ref{lem:factsaboutlimit}.
 \end{proof}


\section{The limit of the critical random graph}
\label{sec:random_graph_limit}

Recall that we are interested in $G(n,p)$ with $p= n^{-1} + \lambda n^{-4/3}$, for $\lambda \in \R$.  
We will begin by recalling some more details of Aldous' limit result from \cite{aldous97brownian}.  His principal tool is the so-called \emph{breadth-first walk} on $G(n,p)$.  This is very similar to our depth-first walk, except that the vertices are considered in a different order.  The \emph{breadth-first ordering} $v_0, v_1, \ldots, v_{n-1}$ on the vertices of the graph is obtained as follows.  (We deliberately use the same notation as in our definition of the depth-first ordering.)  For $i\ge 0$, we define the ordered set $\so_i$ of open vertices at time $i$, and the set $\mathcal A_i$ of the vertices that have already been explored at time $i$. We say that a vertex $u$ has been \emph{seen} at time $i$ if $u\in \so_i \cup \mathcal A_i$.  Let $c_i$ be a counter which keeps track of how many components we have looked at so far.
\begin{steplist}
\item[{\sc Initialization}] Set $\so_0=(1)$, $\mathcal A_0=\emptyset$, $c_0 = 1$. 
\item[{\sc Step $i$}\hspace{38pt}] ($0 \leq i \leq n-1$): Let $v_{i}$ be the first vertex of $\so_i$ and let $\mathcal N_i$ be the set of neighbors of $v_i$ in $[n]\setminus (\mathcal A_i\cup \so_i)$. Set $\mathcal A_{i+1} = \mathcal A_i\cup \{v_i\}$.  Construct $\so_{i+1}$ from $\so_i$ by removing $v_i$ from the start of $\so_i$ and affixing the elements of $\mathcal N_i$ in increasing order to the {\em end} of $\so_i\setminus\{v_i\}$.  If now $\so_{i+1} = \emptyset$, add to it the lowest-labeled element of $[n] \setminus \mathcal A_{i+1}$ and set $c_{i+1} = c_i + 1$.  Otherwise, set $c_{i+1} = c_i$.  
\end{steplist}
The only difference between this procedure and the one introduced in Section \ref{dfs} is that the word ``start'' has been changed to ``end'' (italicized above). 
Now define $Y_n(i) = |\so_i \setminus \{v_i\}| - (c_i - 1)$.  Then $(Y_n(i), 0 \leq i < n)$ is called the breadth-first walk on the graph.  It is straightforward to see that $(Y_n(i), 0 \leq i < n)$ attains a new minimum every time that $v_i$ is the root of a new component.  This enables us to interpret component sizes as excursions above past minima of the breadth-first walk.  Aldous proved that
\begin{equation} \label{aldous}
n^{-1/3} (Y_n(\fl{n^{2/3}t}), t \geq 0) \convdist (W^{\lambda}(t), t \geq 0)
\end{equation}
as $n \to \infty$ in $\mathbb{D}(\R^+, \R^+)$, with convergence (as is usual) uniform on compact time-intervals.  Here,
\[
W^{\lambda}(t) = W(t) + \lambda t - \frac {t^2}2,
\]
where $W$ is a standard Brownian motion.  It is not hard to see that the breadth-first and depth-first walks are interchangeable here, and that the identical result holds for the depth-first walk.  Since we do not actually need this result, we will not go further into details here.

Now define 
\[
B^{\lambda}(t) = W^{\lambda}(t) - \min_{0 \leq s \leq t} W^{\lambda}(s),
\]
the reflecting process. The excursions of this process correspond to ``components" of the limiting graph.    As stated by Aldous, there is an inhomogeneous excursion measure associated with this $B^{\lambda}$, in the same way as It\^o's excursion measure is associated to a reflecting Brownian motion.  Recall from (\ref{excursionspace}) that $\mathcal{E}$ is the space of continuous excursions of finite length.

Denote It\^o's measure by $\N$ and the excursion measure associated to $B^{\lambda}$ by $\N^{\lambda}_t$ (i.e.\ for excursions starting at time $t$).  Then, as argued by Aldous in the proof of his Lemma 26, we can calculate the density of $\N_t^{\lambda}$ with respect to $\N$.  For clarity, we will repeat his argument here.  Firstly note that $\N_t^{\lambda} = \N_0^{\lambda - t}$ and so it suffices to find $\N_0^{\lambda}$ for all $\lambda \in \R$.

Write $W = (W(t), 0 \leq t \leq \sigma)$ for the canonical process under $\N$.  Then by the Cameron--Martin--Girsanov formula \cite{RoWi2000, ReYo2004}, applied under $\N$,
\[
\frac{d \N_0^{\lambda}}{d \N} = \exp \left( \int_0^{\sigma} (\lambda - s) d W(s) - \frac{1}{2} \int_0^{\sigma} (\lambda - s)^2 ds \right).
\]
By integration by parts, we have
\[
\int_0^{\sigma} (\lambda - s) dW(s) = \int_0^{\sigma} W(s) ds,
\]
the area under the excursion $W$.  So we can re-write
\[
\frac{d \N_0^{\lambda}}{d \N} = \exp \left( \int_0^{\sigma} W(s) ds - \frac{1}{6} ((\sigma - \lambda)^3 + \lambda^3)) \right).
\]
We know that we can define a normalized excursion measure $\N(\,\cdot\,| \sigma = x)$ for each $x > 0$, which is, in fact, a probability measure.  There is a corresponding probability measure $\N^{\lambda}_0(\,\cdot\, | \sigma = x)$ which, for $\mathcal{B} \subset \mathcal{E}$ a Borel set, is determined by
\[
\N^{\lambda}_0 [\mathbbm{1}_{\mathcal{B}} | \sigma = x]
= \frac{\N \left[ \exp \left( \int_0^x W(s) ds \right) \mathbbm{1}_{\mathcal{B}} \big| \sigma = x \right]}
           {\N \left[  \exp \left( \int_0^x W(s) ds \right)  \big| \sigma = x \right]}.
\]
Note that this quantity is independent of $\lambda$. By Brownian scaling,
\[
\N \left[  \exp \left( \int_0^x W(s) ds \right)  \bigg| \sigma = x \right]
= \N \left[ \exp \left( x^{3/2} \int_0^1 W(s) ds \right) \bigg| \sigma = 1 \right]
= \E{\exp\left(x^{3/2} \int_0^1 e(s) ds\right)}
\]
where $(e(s), 0 \leq s \leq 1)$ is a standard Brownian excursion under $\mathbb{E}$.  Similarly, for any suitable test function $f$ of the excursion,
\[
\N_0^{\lambda}[f(W(s), 0 \leq s \leq x) | \sigma = x] = \frac{\E{f(\sqrt{x} e(s/x), 0 \leq s \leq 1)  \exp\left(x^{3/2} \int_0^1 e(s) ds\right)} } {\E{\exp\left(x^{3/2} \int_0^1 e(s) ds\right)}  }.
\]

Putting all of this together, we see that the inhomogenity of the excursion measure lies entirely in the selection of the length of the excursion.  So to give a complete description, we just need to determine
$\N_0^{\lambda}(\sigma \in dx)$.  We know that $\N(\sigma \in dx) = (2 \pi)^{-1/2} x^{-3/2}
 dx$ and so
\[
\N_0^{\lambda}(\sigma \in dx) = (2 \pi)^{-1/2} x^{-3/2} \exp\left(-\frac{1}{6}((x - \lambda)^3 + \lambda^3)\right) \N \left[\exp \left( \int_0^x W(s) ds \right) \bigg| \sigma = x \right]dx.
\]
To recapitulate: the excursion measure at time $t$ picks an excursion length according to $\N_0^{\lambda-t}(\sigma \in dx)$.  Then, given $\sigma = x$, it picks a tilted Brownian excursion of that length. This is the crucial fact that allows us to use Theorem~\ref{thm:gh_connectedx} and the results of Section~\ref{sec:connected_components} about the limit of connected components. It is not surprising that it holds, however, since the components of $G(n,p)$ likewise have the property that one can first sample the size and then, given the size, sample a connected component of that size. 

Let $\mathbf{\mathcal{C}}^{n} = (\mathcal{C}_1^{n}, \mathcal{C}_2^{n}, \ldots)$ be the components of the random graph $G(n,p)$ with $p=1/n + \lambda n^{-4/3}$, in decreasing order of their sizes, $Z_1^{n}\ge Z_2^{n}\ge \ldots$ respectively. Let $\mathbf{Z}^{n} = (Z_1^{n},Z_2^{n}, \ldots)$. As a consequence of (\ref{aldous}),  \citet{aldous97brownian} proves that 
\begin{equation}\label{eq:aldous_sizes}
n^{-2/3}\mathbf{Z}^{n} \convdist \mathbf{Z},
\end{equation}
where $\mathbf{Z}$ is the ordered sequence of excursion lengths of $B^{\lambda}$ and convergence is in $\ell_{\searrow}^2$.
Let $\mathbf{M}^{n} = (M_1^{n}, M_2^{n}, \ldots)$ be the sequence of metric spaces corresponding to these components. Recall the definition of $\mathcal M^{(\sigma)}$ from the start of Section \ref{subsec:limob}.  We next state a more precise version of Theorem \ref{thm:main}. 

\begin{thm}\label{thm:4metric}
As $n \to \infty$,
\[
(n^{-2/3} \mathbf{Z}^{n}, n^{-1/3} \mathbf{M}^{n}) \convdist (\mathbf{Z}, \mathbf{M}),
\]
where $\mathbf{M} = (M_1, M_2, \ldots)$ is a sequence of metric spaces such that, conditional on $\mathbf{Z}$, $M_1, M_2, \ldots$ are independent and $M_i \equidist \mathcal{M}^{(Z_i)}$.  Convergence in the second co-ordinate here is in the metric specified by (\ref{4metric}). 
\end{thm}

In proving Theorem \ref{thm:4metric}, we need one additional result, on the expected height of the tilted trees $\tilde{T}_m^p$ introduced in Section~\ref{dfs}. 
This lemma is essentially what allows us to use the distance (\ref{4metric}), rather than product convergence. 
\begin{lem}\label{tiltheight}
Let $p = 1/n + \lambda n^{-4/3}$. There exists a universal constant $M>0$ such that for all $n$ large enough that $1/(2n) < p < 2/n$ and $p < 1/2$, and all $1 \leq m \leq n$, 
\[
\Ec{\|\tilde{T}_m^p\|^4} \leq M\cdot \max(m^6 n^{-4},1) \cdot m^2.
\]
\end{lem}
Before we proceed with the proof, note that the bound in Lemma~\ref{tiltheight} tells us that tilted trees of size of order $n^{2/3}$ behave more or less like uniform trees.  
(See the moments of the height $\|T^m\|$ of uniform trees $T^m$ in \cite{ReSz1967,FlOd1982,FlGaOdRi1993}.) Trees of size much larger than $n^{2/3}$ are much more influenced by the tilting (as witnessed by the factor $m^6n^{-4}$). 
\begin{proof}
We assume throughout that $m \geq 2$. For any $x > 0$ and $\alpha > 0$, we have 
\begin{equation}\label{firstwrite}
\pc{\|\tilde{T}_m^p\| > xm^{1/2}} \leq \Cprobc{\|\tilde{T}_m^p\|> xm^{1/2}}{a(\tilde{T}_m^p) \leq \alpha x^2m^{3/2}} + \pc{a(\tilde{T}_m^p) > \alpha x^2m^{3/2}}.
\end{equation}
We will bound each of the terms on the right-hand side of (\ref{firstwrite}) then integrate over $x$ to obtain the desired bound on $\Ec{\|\tilde{T}_m^p\|^4}$. (We will optimize our choice of 
$\alpha$ later in the proof.) The intuition is that when $a(\tilde{T}_m^p)$ is not too large, the distribution of $\tilde{T}_m^p$ is not too different from that of the uniformly random labeled rooted tree $T_m$, 
and so we should be able to use pre-existing bounds on the tails of $\|T_m\|$. On the other hand, we have already proved (c.f.~Lemma \ref{lem:conv_moments_area}) bounds that 
will allow us to control the probability that $a(\tilde{T}_m^p)$ is large. We now turn to the details. 

Let $q = \max(m^{-3/2},p)$. By Markov's inequality and the definition of $\tilde{T}_m^p$ we have 
\begin{align}
\pc{a(\tilde{T}_m^p) > \alpha x^2m^{3/2}} 	
& \leq \frac{\Ec{(1-q)^{-a(\tilde{T}_m^p)}}}{(1-q)^{-\alpha x^2m^{3/2}}} \nonumber\\
& \leq \frac{\E{((1-p)(1-q))^{-a(T_m)}}}{(1-q)^{-\alpha x^2m^{3/2}}} \nonumber\\
& \leq \frac{\E{(1-q)^{-2a(T_m)}}}{(1-q)^{-\alpha x^2m^{3/2}}}. \label{mi_bigarea}
\end{align}
Let $c=2m^{3/2}/n$, so that $cm^{-3/2}/4 < p < cm^{-3/2}$, and observe that $q \leq \delta m^{-3/2}$ for $\delta:=\max(c,1)$. By Lemma~\ref{lem:conv_moments_area}, there exist absolute constants $K, \kappa>0$ such that 
\[
\sup_{m \geq 1} \Ec{(1-q)^{-2a(T_m)}} \leq K e^{4\kappa\delta^2}.
\]
Furthermore, since $qm^{3/2} \geq \delta/4$, (\ref{mi_bigarea}) yields 
\begin{equation}\label{firstpart}
\pc{a(\tilde{T}_m^p) > \alpha x^2m^{3/2}} \leq K e^{4\kappa\delta^2 - \alpha x^2 \delta/4} \leq K e^{-\alpha x^2 \delta/8},
\end{equation}
for all $x$ such that $x^2\geq 32\kappa \delta/\alpha$. For $x \geq \sqrt{8\ln (2K)/(\alpha \delta)}$, so that $e^{-\alpha x^2\delta/8} \leq 1/2$, since 
$p < cm^{-3/2}$ and $(1-p)^{-1/p} <e^2$, we also have 
\begin{align}
\Cprobc{\|\tilde{T}_m^p\|\geq xm^{1/2}}{a(\tilde{T}_m^p) \leq \alpha x^2m^{3/2}} 	
& \leq \frac{(1-p)^{-\alpha x^2 m^{3/2}}\cdot \pc{\|T_m\| \geq xm^{1/2}}}{\pc{a(\tilde{T}_m^p) \leq \alpha x^2m^{3/2}}} \nonumber\\
& \leq 2e^{2c\alpha x^2}\cdot \pc{\|T_m\| \geq xm^{1/2}}. \label{forluczak}
\end{align}
We can now use tail bounds on the height of uniform labeled trees. {\L}uczak \cite[][Corollary~1]{Luczak1995} provides a uniform tail bound on $\|T_m\|$: for some universal constant $K'$, and all integers $m\ge 1$,
\[ 
\pc{\|T_m\|\ge x m^{1/2}} \le K' x^3 e^{-x^2/2}, 
\]
and so taking $\alpha^{-1} = 8\delta$, (\ref{forluczak}) yields
\begin{equation}\label{secondpart}
\Cprobc{\|\tilde{T}_m^p\|\geq xm^{1/2}}{a(\tilde{T}_m^p) \leq \alpha x^2m^{3/2}} \leq 2K'x^3e^{(2c\alpha - 1/2)x^2} \leq 2K'x^3 e^{-x^2/4}. 
\end{equation}
Notice that our requirements that $x^2 \geq 32\kappa\delta/\alpha$ and $x^2 \geq 8\ln (2K)/(\alpha \delta)$ now reduce to 
$x \geq 16\kappa^{1/2}\delta$ and $x \geq 8\sqrt{\ln (2K)}$. So, in particular, setting $L = 16\kappa^{1/2}+8 \sqrt{\ln (2K)}$, and recalling that $\delta=\max(c,1)$, it suffices to require 
that $x \geq L \delta$. 

For all such $x$, combining (\ref{firstpart}) and (\ref{secondpart}) with (\ref{firstwrite}) and substituting in the value of $\alpha$ then yields 
\[
\pc{\|\tilde{T}_m^p\| > xm^{1/2}} \leq e^{-x^2/64} + 2K'x^3 e^{-x^2/4} \leq \max(4K'x^3,2)e^{-x^2/64}. 
\]
Writing $\E{X}=\int_0^{\infty} \p{X\ge t} dt$, we then have 
\begin{align*}
\Ec{\|\tilde{T}_m^p\|^4}  \leq m^2 \int_0^{m^2} \pc{\|\tilde{T}_m^p\|^4 > xm^2}dx 
\leq m^2 L^4\delta^4 + m^2 \int_{L^4\delta^4}^{m^2} \max(4K'x^{3/4},2)e^{-x^{1/2}/64} dx.
\end{align*}
So, since $\delta=\max(c,1)$ with $c=2m^{3/2}/n$, we have $\Ec{\|\tilde{T}_m^p\|^4}\leq m^2 M (\max(m^{3/2}/n,1))^4$, for some absolute constant $M > 0$, as required. 
\end{proof}

\begin{proof}[Proof of Theorem \ref{thm:4metric}]
	In the random graph $G(n,p)$, conditional on $\mathbf{Z}^{n}$, the components $M_1^{n}, M_2^{n}, \ldots$ are independent and
	\[
	M_i^{n} \equidist G_{Z_i^{n}}^{p}
	\]
where as above, $p= n^{-1} + \lambda n^{-4/3}$.  Note that $n p \to 1$ as $n \to \infty$.  
By (\ref{eq:aldous_sizes}) and Skorohod's representation theorem, there exists a probability space and random variables
$\tilde{\mathbf{Z}}^{n}, \tilde{\mathbf{M}}^{n}, n \geq 1$ and $\tilde{\mathbf{Z}}, \tilde{\mathbf{M}}$ defined on that space such that $(\tilde{\mathbf{Z}}^{n},  \tilde{\mathbf{M}}^{n} ) \equidist (\mathbf{Z}^{n},  \mathbf{M}^{n})$,  $n\ge 1$, and $(\tilde{\mathbf{Z}}, \tilde{\mathbf{M}}) \equidist (\mathbf{Z}, \mathbf{M})$ with $n^{-2/3} \tilde{\mathbf{Z}}^{n} \to \tilde{\mathbf{Z}}$ a.s.  But then the convergence $(n^{-2/3} \mathbf{Z}^{n}, n^{-1/3} \mathbf{M}^{n}) \convdist (\mathbf{Z}, \mathbf{M})$ in the product topology follows immediately from Theorem~\ref{thm:gh_connectedx}.  
We can, and will hereafter assume, again by applying Skorohod's theorem, that $(n^{-2/3}Z_i^{n},n^{-1/3}M_i^{n}) \to (Z_i,M_i)$ almost surely for all $i$. 
	It remains to prove convergence in distribution in the metric specified by (\ref{4metric}). In doing so, we will need to use the $\oDFS$ procedure. For any $n$ and $i$ for which $M_i^{n}$ is defined, 
	we may view $M_i^{n}$ as a finite connected graph; this graph is uniquely specified (up to isomorphism) by $M_i^{n}$. When we write $\oDFS(M_i^{n})$ we mean the $\oDFS$ procedure run on a uniformly random labelling of the graph corresponding to $M_i^{n}$. 

To prove convergence in the metric specified by (\ref{4metric}), we first observe that for any sequences of metric spaces $\mathbf{A},\mathbf{B}$ and any integer $N \geq 1$, we have 
	\[
	d(\mathbf{A},\mathbf{B}) \leq \pran{\sum_{i=1}^{N-1} \dgh(A_i,B_i)^4}^{1/4} + \pran{\sum_{i=N}^{\infty} \dgh(A_i,B_i)^4}^{1/4}.
	\]
	Since we have already established convergence in the product topology, to complete the proof it thus suffices to show that for all $\epsilon > 0$, 
	\begin{equation}\label{4metric2show}
	\lim_{N \rightarrow \infty} \limsup_{n \rightarrow \infty} \p{\sum_{i=N}^{\infty} \dgh(n^{-1/3}M_i^{n},M_i)^4 >\eps} = 0.
	\end{equation}
	As earlier, we write $\|\cdot\|$ for the height of a rooted tree or the supremum of a finite excursion. For any $i$ and $n$, 
	we may bound use the bound 
	\begin{equation}\label{normbound}
	\dgh(n^{-1/3}M_i^{n}, M_i)^4 \le 16(n^{-4/3}\|\tilde T_i^{n}\|^4 + \|\tilde e^{(Z_i)}\|^4),
	\end{equation}
	where $\tilde T_i^{n}$ is the depth-first tree corresponding to \oDFS$(M_i^{n})$ started at its smallest vertex, and $\tilde e^{(Z_i)}$ is the excursion corresponding to $M_i$. 
	Now let 
	\[
	\Xi_i^{n} = \|\tilde T_i^{n}\|^4 \cdot (Z^{n}_i)^{-2}.
	\]
	By Brownian scaling, given the length $Z_i$, we have that $\|\tilde e^{(Z_i)}\|^4=Z_i^2 \cdot \|\tilde{e}_i\|^4$, where $\{\tilde{e}_i, i\ge 1\}$ are independent and identically distributed copies of $\tilde e$,
	a tilted excursion of length one, and which are independent of $\{Z_i, i\ge 1\}$. Combining the preceding equalities with (\ref{normbound}), 
	we thus have 
	\[
	\sum_{i=N}^{\infty} \dgh\pran{n^{-1/3} M_i^{n}, M_i}^4 \leq 16 \sum_{i=N}^{\infty} \pran{\Xi_i^{n} \pran{n^{-2/3} Z^{n}_i }^2 + Z_i^2\|\tilde{e}_i\|^4}.
	\]
	Next, given $\delta>0$ write $N_{\delta} = N_{\delta}(\mathbf{Z})$ for the smallest $N$ such that $Z_N < \delta$; 
	$N$ is almost surely finite since $\mathbf{Z}$ is almost surely an element of $\ell^2_{\searrow}$. 
	For any $\delta > 0$ and all $n,N$, setting $\epsilon_1=\epsilon/16$ we then have 
	\[
	\p{\sum_{i=N}^{\infty} \dgh(n^{-1/3}M_i^{n}, M_i)^4>\eps} \leq \p{ \sum_{i>N_{\delta}} \pran{\Xi_i^{n} \pran{n^{-2/3}Z^{n}_i}^2 + Z_i^2\|\tilde{e}_i\|^4} > \epsilon_1} + \p{N_{\delta}>N}.
	\]
	Since $\lim_{N \rightarrow \infty} \p{N_{\delta}>N}=0$, and the first probability on the right-hand side of the preceding inequality does not depend on $N$, we thus have 
	\[
	\lim_{N \rightarrow \infty} \limsup_{n \rightarrow \infty} \p{\sum_{i=N}^{\infty} \dgh\pran{n^{-1/3}M_i^{n}, M_i}^4>\epsilon} \leq 
	\limsup_{n \rightarrow \infty} \p{\sum_{i>N_{\delta}} \pran{ \Xi_i^{n} n^{-4/3}(Z^{n}_i)^2 + Z_i^2\|\tilde{e}_i\|^4}>\epsilon_1}.
	\]
	Since this holds for any $\delta>0$ and the left-hand side does not depend on $\delta$, we then obtain 
	\begin{equation}\label{2partsbound}
	\lim_{N \rightarrow \infty} \limsup_{n \rightarrow \infty} \p{\sum_{i=N}^{\infty} \dgh\pran{n^{-1/3}M_i^{n},M_i}^4>\epsilon} \leq 
	\lim_{\delta \downarrow 0}\limsup_{n \rightarrow \infty} \p{\sum_{i>N_{\delta}} \pran{ \Xi_i^{n} \left(\frac{Z^{n}_i}{n^{2/3}} \right)^2 + Z_i^2\|\tilde{e}_i\|^4}>\epsilon_1}.
	\end{equation}

	But it follows from Corollary 2 and Lemma 14 (b) of \cite{aldous97brownian} that for all $\gamma > 0$, 
	\[
	\lim_{\delta \downarrow 0}\limsup_{n \rightarrow \infty} \p{\sum_{i>N_{\delta}} \pran{\frac{Z^{n}_i}{n^{2/3}}}^2 > \gamma} = 0, \qquad\mbox{and}\qquad
	\lim_{\delta \downarrow 0} \p{\sum_{i>N_{\delta}} Z_i^2 > \gamma} = 0, 
	\]
	from which we may complete the proof straightforwardly. 
	First recall that $\tilde{e}_i$, $i\ge 1$ are independent and identically distributed tilted excursions of length one.  Moreover, $\E{\|\tilde e_i\|^4}<\infty$, using the change of measure in the definition of $\tilde e_i$ and the Gaussian tails for the maximum
	 $\|e\|$ of a standard Brownian excursion $e$ \cite{kennedy1976}.
	Let $\epsilon_2=\epsilon_1/2$, and choose $\delta > 0$ small enough that 
	\[
	\p{\sum_{i>N_{\delta}} Z_i^2 \geq \frac{\epsilon_2^2}{2\E{\|\tilde{e}\|^4}}} \leq \frac{\epsilon_2}2.
	\]
	Then, by Markov's inequality, we have 
	\begin{align*}
	\p{\sum_{i>N_{\delta}} Z_i^2\|\tilde{e}_i\|^4 \geq \epsilon_2} 
	& \leq \frac{\epsilon_2}{2} 
	+ \Cprob{\sum_{i>N_{\delta}} Z_i^2\|\tilde{e}_{i}\|^4 \geq \epsilon_2}
	{\sum_{i>N_{\delta}} Z_i^2 < \frac{\epsilon_2^2}{2\E{\|\tilde{e}\|^4}}} \\
	& \leq \frac{\epsilon_2}{2} 
	+ \frac{1}{\epsilon_2} \ExpC{\sum_{i>N_{\delta}} Z_i^2\|\tilde{e}_i\|^4}{ \sum_{i>N_{\delta}} Z_i^2 < \frac{\epsilon_2^2}{2\E{\|\tilde{e}\|^4}}} \le \epsilon_2,
	\end{align*}
	since $\{\tilde{e}_i, i\ge 1\}$ is independent of the set of lengths $\{Z_i, i\ge 1\}$.
	Since $\epsilon=32 \epsilon_2$ was arbitrary, it follows that
	\[
	\lim_{\delta \downarrow 0}  \p{\sum_{i>N_{\delta}} \|\tilde{e}_i\|^4Z_i^2 \geq \epsilon_2} = 0.
	\]
	We may apply an identical argument to bound the terms involving discrete random variables and show that 
	\[
	\lim_{\delta \downarrow 0} \limsup_{n \rightarrow \infty} \p{ \sum_{i>N_{\delta}} \Xi_i^{n} \pran{Z^{n}_i n^{-2/3}}^2 > \epsilon_2} =0,
	\]
	and thereby complete the proof, as long as we can show that $\sup_{i,n} \mathbb{E} [\Xi_i^{n}]<\infty$.  We now prove that this is true.
	Recall that, by definition, $\Xi_i^{n} = \|\tilde T_i^{n}\|^4/(Z^{n}_i)^2$. 
	Recall also that for a given $n$ and $i$, conditional on $Z^{n}_i=m$, the tree
	$\tilde{T}_i^{n}$ is distributed as $\tilde{T}_m^p$ from Section \ref{dfs}. In other words, we have 
	\[
	\pc{\tilde{T}_i^{n} = T} \propto (1-p)^{-a(T)},
	\]
	for each tree $T$ on $[m]$, where by $\tilde{T}_i^{n} = T$ we mean that the increasing map (i.e.\ respecting the increasing order of the labels) between vertices of $\tilde{T}_i^{n}$ and $T$ induces an isomorphism. 

	To bound $\mathbb{E}[\Xi_i^{n}]$, we use Lemma \ref{tiltheight}, together with bounds on the size of the largest component of 
	$G(n,p)$ for $p=1/n+\lambda n^{-4/3}$. For our purposes, the latter bounds are most usefully stated by \citet{nachmias08crit} (see also \cite{pittel01largest,scott06solving}). They proved that for all fixed $\lambda \in \R$, 
	there exist $\gamma=\gamma(\lambda)>0$ and $C=C(\lambda)>1$ such that for all $n$ and for all $x \geq C$, 
	\[
	\pc{Z^{n}_1 \geq xn^{2/3}} \leq e^{-\gamma x^3}.
	\]
	(In fact, the bound in \cite{nachmias08crit} is slightly stronger than this.) For all integer $i\ge 1$, we thus have 
	\begin{align*}
	\mathbb{E}[\Xi_i^{n}]	& \leq \sup_{m \leq C n^{2/3}} m^{-2}\CExpc{\|\tilde T_i^{n}\|^4}{Z_i^{n}=m} + \sum_{m = \lceil Cn^{2/3}\rceil}^n m^{-2} \CExpc{\|\tilde T_i^{n}\|^4 }{Z_i^{n}=m}\cdot \pc{Z_i^{n}=m}\\
				&\leq \sup_{m \leq C n^{2/3}} m^{-2} \Ec{\|\tilde T_m^p\|^4}
						+ \sum_{m = \lceil Cn^{2/3}\rceil}^n m^{-2} \Ec{\|\tilde{T}_m^p\|^4}\cdot \pc{Z_1^{n}\geq m}\\
				&\leq \sup_{m \leq C n^{2/3}} m^{-2} \Ec{\|\tilde{T}_m^p\|^4} 
						+ \sum_{m = \lceil Cn^{2/3}\rceil}^n m^{-2}\Ec{\|\tilde{T}_m^p\|^4}\cdot e^{-\gamma m^3 n^{-2}}.
	\end{align*}
	Applyling Lemma \ref{tiltheight} to the last expression, we immediately obtain 
	\begin{align*}
	\mathbb{E}[\Xi_i^{n}]	& \leq M C^{6} + \sum_{m = \lceil Cn^{2/3}\rceil}^n M\cdot m^{6}n^{-4} \cdot e^{-\gamma m^3n^{-2}}
	\end{align*}
	which is uniformly bounded in both $i\ge 1$ and $n\ge 1$, as required.
\end{proof}

We finally turn to the diameter of the random graph $G(n,p)$.  Note that, in order to prove convergence in the metric   (\ref{4metric}), we in fact proved that for all $\epsilon > 0$, 
\begin{equation*}
\lim_{N \rightarrow \infty} \limsup_{n \rightarrow \infty}\p{\sum_{i\ge N} n^{-4/3}\|\tilde T_i^{n}\|^4 \geq \epsilon} = 0 \quad \mbox{ and } \quad \lim_{N \rightarrow \infty} \p{\sum_{i\ge N}\|\tilde e^{(Z_i)}\|^4\ge \epsilon} = 0. 
\end{equation*}
Since $\diamc{M_i^{n}} \leq 2\|\tilde T_i^{n}\|$ and $\diamc{M_i} \leq 2\|\tilde e^{(Z_i)}\|$, we must thus have 
\begin{equation}\label{4met:strong} 
\lim_{N \rightarrow \infty} \limsup_{n \rightarrow \infty}\p{\sup_{i \geq N}n^{-4/3}\diamc{M_i^{n}}^4 > \epsilon} = 0 
\quad \mbox{ and } \quad 
\lim_{N \rightarrow \infty}\p{\sup_{i \geq N}\diamc{M_i}^4 > \epsilon} = 0, 
\end{equation} 
which allows us to prove the   convergence of the diameters. 
For $i=1,2,\ldots$ let $D^{n}_i = \diamc{M_i^{n}}$ if $\mathbf{M}^{n}$ has 
at least $i$ components, and $D^{n}_i=0$ otherwise, and let  
\[
D^{n} = \max_{i\ge 1}\left\{\diamc{M_i^{n}}\right\}
\] 
denote the diameter of $G(n,p)$. We remark that Aldous discusses the diameter of continuum metric spaces and the Brownian CRT in particular \cite[][Section 3.4]{aldous91crt2}. We can now prove Theorem~\ref{thm:diam_gnp} stated in the introduction, which claims that $n^{-1/3} D^{n}\convdist D$, for a random variable $D\ge 0$ with finite mean and absolutely continuous distribution. 

\begin{proof}[Proof of Theorem~\ref{thm:diam_gnp}]
For each $i=1,2,\ldots$ let $D_i = \diam{M_i}$, and let $D = \sup\{\diam{M_i}:i \geq 1\}$. 
Observe that for any two metric spaces $M$ and $N$, 
$|\diam{M}-\diam{N}| \leq 2\dgh(M,N)$. For fixed $i$, the claimed convergence is immediate from Theorem \ref{thm:4metric} since 
$$|n^{-1/3}\diamc{M_i^{n}}-\diam{M_i}| \leq 2\dgh(n^{-1/3}M_i^{n},M_i)$$
and $n^{-1/3}M_i^{n} \to M_i$ in distribution in the Gromov--Hausdorff distance. Also $\diam{M_i}$ is a non-negative random variable with finite mean since $\diam{M_i}$ is at most the diameter of the underlying continuum random tree: $\diam{M_i}\le 2 \|\tilde e_i\|\sqrt{Z_i} $, where $\tilde e_i$ is a tilted excursion of length one.  
It follows immediately that for any fixed $N$, 
\begin{equation}\label{truncation}
n^{-1/3} \max\{D^{n}_i, 1\le i\le N\} \convdist \max\{D_i, 1\le i\le N\}.
\end{equation}
Next, observe that for any $\epsilon > 0$, there exists $c > 0$ such that $\p{D < c} \le \p{D_1 < c} < \epsilon/2$. Since $n^{-1/3} D_1^{n}\convdist D_1$, we must then have that, for all $n$ large enough, 
$$\pc{n^{-1/3} D^{n} < c} \le \pc{n^{-1/3} D_1^{n} < c} < \epsilon/2.$$
Now fix any $\epsilon > 0$ and take $N$ large enough that for all $n$ sufficiently large,
\[
\p{\sup_{i \geq N} n^{-1/3} \diamc{M_i^{n}} \geq c} \leq \epsilon/2 \quad \mbox{ and } \quad \p{\sup_{i \geq N}\diam{M_i} \geq c} \leq \epsilon/2;
\]
such an $N$ must   exist by (\ref{4met:strong}). 
Then for all $n$ sufficiently large, by the preceding equations
\[
\p{D^{n} \neq \max_{1\leq i \leq N} D^{n}_i} \leq \epsilon \qquad \mbox{and}\qquad \p{D \neq \max_{1 \leq i \leq N} D_i} \leq \epsilon.
\]
Since $\epsilon > 0$ was arbitrary, combining these inequalities with (\ref{truncation}) yields that $n^{-1/3} D^{n} \convdist D$, as required.  It follows straightforwardly from the behavior of the tail of the sequence $(D_i, i \ge 1)$ that $D$ has an absolutely continuous distribution.
Finally, the fact that $\E{D} < \infty$ is a direct consequence of \cite{addario08diam}, Theorem 1.
\end{proof}


\section{Concluding remarks}

In this paper we have proved that it is possible to define a scaling limit for critical random graphs using random continuum metric spaces. This gives us a systematic way to consider a great many questions about distances in critical random graphs. In particular, it allows us to prove that critical random graphs have a diameter of order $n^{1/3}$ which is not concentrated around its mean. Focussing just on the diameter, there are now several questions which might deserve another look: what is the probability that the largest component achieves the diameter?  What is the distribution of the (random) point $\lambda \in \R$ (where $p=1/n+\lambda n^{-4/3}$) at which the diameter of the random graph maximized? What is the distribution of this diameter? 

The proof of our main result relies on a careful analysis of a depth-first exploration process of the graph which yields a ``canonical'' spanning forest and a way to add surplus edges according to the appropriate distribution. The forest is made of non-uniform trees that are biased in favor of those with a large area. 
In the limit, these trees rescale to continuum random trees encoded by tilted Brownian excursions. We have limited our analysis of these excursions to a minimum, but it seems likely that much more can be said, which might in turn yield results for the structure of the graphs or the behavior of other graph exploration~algorithms.

In this paper, we have very much relied upon the depth-first viewpoint.  Gr\'egory Miermont~\cite{Miermont} has suggested that, at least intuitively, there should be an analogous breadth-first approach to the study of a limiting component, in which one might think of the shortcuts as being made ``horizontally'' across a generation rather than ``vertically'' along paths to the root.  The advantage of the depth-first walk is that it converges to the same excursion as the height process of the depth-first tree.  The rescaled breadth-first walk, however, converges to the same limit as the rescaled \emph{height profile} (i.e.\ the number of vertices at each height) of a ``breadth-first tree'', which contains less information and, in particular, does not code the structure of that tree.  As a result, it seems that it would be much harder to derive a metric space construction of a limiting component using the breadth-first viewpoint.  It may, nonetheless, be the case that the breadth-first perspective is better adapted to answering certain questions about the limiting components where only the profile matters.

In a companion paper \cite{AdBrGo2009b}, we will describe an alternative construction of the limit object which has the cycle structure of connected components at its heart: a connected component may be described as a multigraph (which gives the cycles), onto which trees are pasted. Together with the results of this paper, the latter perspective yields many limiting distributional results about sizes and lengths in critical random graphs.

\small
\setlength{\bibsep}{.3em}
\bibliographystyle{plainnat}
\bibliography{bib_clcrg,bib_l,bib_c,bib_n}

\end{document}